\documentclass[11pt, twoside, reqno]{amsart} 	

\usepackage[colorinlistoftodos]{todonotes}
\usepackage[T1]{fontenc}
\usepackage[utf8]{inputenc}

\usepackage{color}
\usepackage{geometry}                		
\usepackage{graphicx}			
\usepackage{amscd, changepage, cancel}
\usepackage[normalem]{ulem}
\usepackage{faktor}

\usepackage{enumitem}
\usepackage{comment, csquotes}
\usepackage{amssymb}
\usepackage{amsfonts}
\usepackage{amsthm}
\usepackage{amsmath}
\usepackage{mathabx}
\usepackage{bbm}
\usepackage{xcolor}
\usepackage{tikz}
\usepackage{subcaption}

\usepackage{fancyhdr}

\usepackage{etoolbox}

\definecolor{forestgreen}{rgb}{0.13, 0.55, 0.13}
\definecolor{greenao}{rgb}{0.0, 0.5, 0.0}
\definecolor{dimgray}{rgb}{0.41, 0.41, 0.41}
\definecolor{afb}{rgb}{0.36, 0.54, 0.66}
\definecolor{bncs}{rgb}{0.0, 0.53, 0.74}
\definecolor{purple}{rgb}{0.47, 0.32, 0.66}
\definecolor{tangerine}{rgb}{0.95, 0.52, 0.0}
\definecolor{darklavender}{rgb}{0.45, 0.31, 0.59}
\definecolor{darkorchid}{rgb}{0.6, 0.2, 0.8}
\definecolor{do}{rgb}{0.6, 0.2, 0.8}
\definecolor{bdf}{rgb}{0.19,0.55, 0.91}

\newcommand{\MI}[1]{\textcolor{teal}{#1}}

\DeclareMathOperator{\Rea}{Re}
\DeclareMathOperator{\I}{Im}

\newcommand{\R}{\mathbb{R}}

\newcommand{\la}{\langle}
\newcommand{\ra}{\rangle}

\newcommand{\lb}{\lambda}

\newcommand{\ep}{\epsilon}

\newcommand{\wh}{\widehat}

\newcommand{\p}{\partial}
\newcommand{\px}{\partial_x}
\newcommand{\pv}{\partial_v}
\newcommand{\pt}{\partial_t}
\newcommand{\cd}{\cdot}

\newcommand{\lp}{\left(}
\newcommand{\rp}{\right)}

\newcommand{\lefb}{\left[}
\newcommand{\rigb}{\right]}

\newcommand{\bu}{\bar{u}}
\newcommand{\bv}{\bar{v}}
\newcommand{\bU}{\bar{U}}

\newcommand{\W}{\mathcal{W}}
\newcommand{\bW}{\widebar{\W}}
\newcommand{\uap}{u_{app}}
\newcommand{\buap}{\bu_{app}}
\newcommand{\bLU}{\widebar{LU}}

\numberwithin{equation}{section}
\allowdisplaybreaks

\newtheorem{thm}{Theorem}[section]
\newtheorem{definition}{Definition}[section]
\newtheorem{lemma}{Lemma}[section]

\newtheorem{prop}{Proposition}[section]
\usepackage{hyperref}
\hypersetup{
    colorlinks=true,
    linkcolor=blue,
    citecolor=forestgreen,
    filecolor=blue,      
    urlcolor=black,
}

\makeatletter
\def\@tocline#1#2#3#4#5#6#7{\relax
  \ifnum #1>\c@tocdepth % then omit
  \else
    \par \addpenalty\@secpenalty\addvspace{#2}%
    \begingroup \hyphenpenalty\@M
    \@ifempty{#4}{%
      \@tempdima\csname r@tocindent\number#1\endcsname\relax
    }{%
      \@tempdima#4\relax
    }%
    \parindent\z@ \leftskip#3\relax \advance\leftskip\@tempdima\relax
    \rightskip\@pnumwidth plus4em \parfillskip-\@pnumwidth
    #5\leavevmode\hskip-\@tempdima
      \ifcase #1
       \or\or \hskip 1em \or \hskip 2em \else \hskip 3em \fi%
      #6\nobreak\relax
    \hfill\hbox to\@pnumwidth{\@tocpagenum{#7}}\par% <---- \dotfill -> \hfill
    \nobreak
    \endgroup
  \fi}
\makeatother

\title{Modified Scattering and Asymptotic Completeness for the Derivative Nonlinear Schr\"odinger equation}

\author[]{Allison Byars}

\date{}

\begin{document}
\begin{abstract}
   We prove a modified scattering and asymptotic completeness for the derivative nonlinear Schr\"odinger equation.  This is the first result proving asymptotic completeness in a quasilinear setting.  Our approach combines the method of testing by wave packets, introduced by Ifrim and Tataru, a bootstrap argument, and the Klainerman–Sobolev vector field method.
\end{abstract}

\maketitle

\tableofcontents

\section{Introduction}
We consider the Cauchy problem for the derivative nonlinear Schr\"odinger
(DNLS) equation
\begin{equation}
\begin{cases}
iu_t+u_{xx}=-i\partial_x(|u|^2u)\\
u(0,x)=u_0(x),
\end{cases}
\tag{DNLS}
\label{dnls}
\end{equation}
 where the unknown is a complex valued function, $u:\mathbb{R}\times \mathbb{R}\rightarrow \mathbb{C}.$

The \eqref{dnls} equation arises as a model of large-wavelength Alfv\'en waves in plasma, as well as in other physical settings. The equation was first derived in 1974 by Mj{\o}elhus \cite{mjoelhus1974application} using the reductive perturbation method. For more physical applications, we refer the interested reader to the following references \cite{cite:champeaux, jenkins2019derivative,cite:physical2,cite:physical3, cite:physical4}. 

\eqref{dnls} is a dispersive equation, which has the dispersion relation $\omega (\xi)=-\xi^2,$ with the group velocity given by $\omega'(\xi)=-2\xi$. \eqref{dnls} also admits a scaling law:
\begin{equation*}
    u_\lambda (t,x):= \sqrt{\lambda}u(\lambda^2t, \lambda x),
\end{equation*}
where if $u$ is a solution, then so is $u_\lambda.$  This scaling gives the critical Sobolev space as $L^2(\R)$. 

\eqref{dnls} is completely integrable, meaning it has an infinite number of conservation laws, a Lax pair \cite{cite:KaupNew}, and an inverse scattering transformation \cite{jenkins2018global}, which are just a few of the many properties completely integrable systems posses. We will list some conservation laws for our equation below, however, it is important to note that the results and techniques employed in this paper do not rely on the system being integrable. 

As mentioned about, \eqref{dnls} has an infinite number of conserved quantities.  In particular, it has one at the level of $H^{s/2}$, for every $s\in \mathbb{N}$ \cite{wyller1989conserved}. Three of the most common conserved quantities are the mass, at the level of $L^2$,
\begin{equation*}
    M(u):= \int |u(x)|^2\, dx,
\end{equation*}
the momentum, at the level of $H^{\frac{1}{2}}$
\begin{equation*}
    P(u):=\int \I(\bar{u}\px u)-\frac{1}{2}|u|^4 \, dx,
\end{equation*}
and the energy, at the level of $H^1$
\begin{equation*}
     E(u):=\int |u_x|^2-\frac{3}{2}|u|^2 \I(\bar{u}u_x)+\frac{1}{2}|u|^6 \, dx.
\end{equation*}
It also has a Hamiltonian structure given by the operator $\mathcal{J}=\px$ with $P$ as the Hamiltonian. This means we can generate the \eqref{dnls} equation using the functional derivative of the momentum: 
\begin{align*}
    u_t=\mathcal{J}\left(\frac{\delta P(u)}{\delta \bu}\right).
\end{align*}

The well-posedness of \eqref{dnls} has been studied extensively. As mentioned above, the scaling critical space is $L^2,$ so when proving well-posedness in $H^s,$ the case $s=0$ is a natural threshold to target.  It is also important to know the regularity where the data-to-solution map is Lipschitz continuous.  When this property holds, the equation behaves more like a semilinear one, and can be handled with a fixed point argument. In contrast, when Lipschitz continuity fails, one can think of the problem as quasilinear, with only continuity holding. For \eqref{dnls}, the map fails to be Lipschitz below $s=1/2$ \cite{cite:Bialess12}. Early work established local well-posedness for $s> 3/2$ using a contraction argument \cite{cite:WP>3/2,cite:WP>3/2.2}, as well as ill-posedness below $s=0$ from self-similar solutions \cite{cite:selfsim1,cite:selfsim2}.  The well-posedness regularity was later lowered to $1/2$ by Takaoka \cite{cite:Taks12}, and global well-posedness for $H^{\frac{1}{2}}$ was obtained in 2022 \cite{bahouri2022global}.
In 2023 Killip, Ntekoume, and Vi\c{s}an studied the well-posedness of \eqref{dnls} below the $s=1/2$ threshold, and were able to prove well-posedness for $1/6<s<1/2$ \cite{cite:WPvisanH16}, and eventually all the way down to $s=0,$ which is the critical space, $L^2$ \cite{cite:1}.

The purpose of this paper is to prove a modified scattering and asymptotic completeness result for \eqref{dnls}. The present work should be viewed as a natural continuation of our earlier analysis of \eqref{dnls}, see \cite{byars2024globaldynamicssmalldata} . In \cite{byars2024globaldynamicssmalldata}, we demonstrated that global solutions satisfy a dispersive estimate under the assumption of small and localized initial data. Gaining an understanding of the global dynamics, beyond the scope of a mere global existence theory, provided the essential groundwork for the developments we pursue in this work. The main result of that paper was as follows:
\begin{thm}[Theorem 1.1 of \cite{byars2024globaldynamicssmalldata}]
\label{dispersive}
    Let $u$ be a solution to \eqref{dnls} with small and localized data, i.e. for $\epsilon \ll 1,$
    \begin{equation}
\|xu_0\|_{H^1}+\|u_0\|_{H^1}\leq \epsilon
    \end{equation}
Then we have the following bound on $Lu$,
\begin{equation}
    \|Lu\|_{H^1_x}\lesssim \epsilon\, \langle t \rangle^{C\epsilon^2},
    \label{Lubound1}
\end{equation}
and the dispersive bounds
\begin{equation*}
    \|u(t,x)\|_{L^\infty_x} +  \|u_x(t,x)\|_{L^\infty_x}\lesssim \epsilon \, \langle t\rangle^{-\frac{1}{2}} \text{ for all } t\in \R,
    \label{dispersiveestimate1}
\end{equation*}
\end{thm}
where $\langle \cdot \rangle$ are the usual Japanese brackets, and $L$ is the vector field 
\begin{align*}
    \langle t\rangle:=\left(1+t^2\right)^{\frac{1}{2}},\,\, L:=x+2ti\px.
\end{align*}

From \cite{cite:1}, it is known that \eqref{dnls} has the following family of solitons. For $\theta\in (0,\frac{\pi}{2})$, define the initial profile
\begin{align*}
    q_0(x,\theta):=&\sqrt{2\sin (2\theta)}\frac{\cosh^3(x-i\theta)}{|\cosh(x-i\theta)|^4}e^{-ix\cot(2\theta)}.
\end{align*}
The corresponding soliton solution is
\begin{align*}
    q(t,x,\theta)=q_0(x+2\cot(2\theta)t,\theta)e^{it\csc^2(2\theta)}.
\end{align*}
The above solitons can also be rescaled, translated and rotated, and hence the problem admits a four-parameter family of solitons. This observation is particularly important, since solitons are known to preserve their shape and therefore do not exhibit dispersive decay. On the other hand, the setting adopted in \cite{byars2024globaldynamicssmalldata} rules out the existence of small solitons.

In this work, we again restrict ourselves to small, localized initial data, for which solitons do not arise. To understand why, we provide the following calculation.
In \cite{cite:1}, it is shown that the $L^2$ norm of this soliton satisfies
\begin{align*}
    \|q_0\|_{L^2_x}^2=\|q\|_{L^2_x}^2=8\theta,
\end{align*}
which is small in $L^2$ for small $\theta.$ However, one also finds
\begin{align*}
   \|xq_0\|_{L^2}, \,\,\, \|q_0'\|_{L^2}\sim \frac{1}{\theta},
\end{align*}
so that $\|xq_0\|_{H^1}\sim \theta+\frac{1}{\theta}$. In particular,
\begin{align*}
    \Vert xq_0\Vert_{H^1}\Vert q_0\Vert_{L^2}\geq 1.
\end{align*}
This is true not only for the above $q_0$ but also for all of its rescaled versions.  This shows that indeed, it is not possible for the solitons to be localized in $H^1$ and small in $L^2$.

\subsection{Why modified scattering?}

After obtaining a global dispersive bound, a natural next question to ask is whether the solution scatters to its corresponding linear solution.

In general, for a linear dispersive equation
\begin{align*}
    i\pt \phi-A\phi=0,
\end{align*}
the solution is given by $\phi(t)=e^{-itA}\phi_0,$ or some initial data $\phi_0$.  We say the nonlinear solution $u$ \emph{scatters}  to a linear solution  in a space $X$ f there exists some linear data $u_+$ so that 
\begin{align*}
\|u(t,x)-e^{-itA}u_+\|_{X}\to 0
\end{align*}
as $t\to \infty.$ Classical examples of scattering include NLS \cite{scattering_nls_1979}, Klein-Gordon \cite{strauss1974scattering}, and a more general, detailed explanation in \cite[section 3.6]{cite:Tao}.

For Schr\"odinger-type equations classical scattering is expressed in the form 
\begin{align*}
\|u(t,x)-t^{-\frac{1}{2}}e^{-\frac{ix^2}{2t}} W(x/t)\|_{X}\to 0,
\end{align*}
where $W$ can be thought as an asymptotic profile. See section \ref{linear section} for more about the linear solution. 

In many cases, when the problem is strongly nonlinear, the chances of having a classical scattering result decrease.  Instead, we aim to prove a modified scattering result complemented by an asymptotic completeness theory for the problem. \emph{Modified scattering} is similar to the classical case, but the asymptotic profile is altered by nonlinear interactions, leading to additional terms, see e.g. \cite{hayashi_naumkin_modscat_nls_1998,cite:NLS, harrop2016mod_scat}.

For \eqref{dnls}, classical scattering cannot occur. Instead, we prove that we have a modified  scattering-type asymptotic behavior
\begin{align*}
\|u(t,x)-t^{-\frac{1}{2}}e^{\frac{ix^2}{4t}-\frac{ix}{2t}|W(x/t)|^2\log t} W(x/t)\|_{X}\to 0
\end{align*}
where we get a logarithmic correction. This claim is supported by our previous result \cite{byars2024globaldynamicssmalldata}, which showed that the nonlinear solution decays at the same rate as the linear one. 

Here, we prove both modified scattering and asymptotic completeness for small-data solutions. Our main theorem is as follows.
\begin{thm}
    Let $u$ be a solution to \eqref{dnls} with small and localized data, i.e. for $\epsilon \ll 1,$
    \begin{equation}
        \|xu_0\|_{H^1}+\|u_0\|_{H^1}\leq \epsilon.
        \label{initials}
    \end{equation}

   \begin{enumerate}[label=(\alph*)]
  \item (Modified Scattering)
   Then there exists some function $W\in H^{1-C\epsilon^2,1-C\epsilon^2}$ such that
    \begin{align}\label{u mod scat}
        u(t,x)=&t^{-\frac{1}{2}}e^{i\frac{x^2}{4t}}W\left(\frac{x}{t}\right)e^{-\frac{ix}{2t}|W\left(\frac{x}{t}\right)|^2\log t}+err_x,\\
        \hat{u}(t,\xi)=&e^{it\xi^2}W(2\xi)e^{-i\xi|W(2\xi)|^2\log t}+err_\xi ,\label{hatu mod scat}
    \end{align}
    where 
    \begin{align*}
        err_x\in O_{L^\infty}(\epsilon (1+t)^{-\frac{3}{4}+C\epsilon^2})\cap  O_{L^2}(\epsilon (1+t)^{-1+C\epsilon^2}),
    \end{align*}
    and 
    \begin{align*}
      err_\xi\in O_{L^\infty}(   \epsilon (1+t)^{-\frac{1}{4}+C\epsilon^2})\cap O_{L^2}(   \epsilon (1+t)^{-\frac{1}{2}+C\epsilon^2}).
    \end{align*}
         \item (Asymptotic Completeness) Let $C$ be a large universal constant.  For each $W$ satisfying 
    \begin{equation*}
        \|W\|_{H^{3/2+C\epsilon^2}(\R)}\ll \epsilon \ll 1, \,\,\, \|vW\|_{H^{3/2+2\delta}_v}\ll\epsilon, \,\,\, \|v^2W\|_{H^{1}_v}\ll\epsilon
    \end{equation*}
    there exists an unique $u_0\in H^{1,1}$ satisfying \eqref{initials} such that \eqref{u mod scat} and \eqref{hatu mod scat} hold for the corresponding solution to \eqref{dnls}.
   \end{enumerate}
   
    \label{mainthm}
\end{thm}

It is important to note that this problem has been studied before in \cite{hayashi1994modified,HayashiNaumkin2013modified}. However, there are two key differences between those works and the present one: the assumptions on $W,$ and the methods used, which are tailored to quasilinear problems.

We define the weighted Sobolev space $H^{s,m}$ to be the space of functions $f$ with regularity $s,$ and spatial decay $m,$ i.e., $f\in H^{s,m}$ if $\la x \ra^mf\in H^s$. The norm is given by
\begin{equation}\label{weighted sobolev space defn}
    \|f\|_{H^{s,m}}:=\|(1+|x|^2)^{m/2}(1-\px^2)^{s/2}f\|.
\end{equation}
In \cite{hayashi1994modified}, the authors assume  $W\in H^{0,4} \cap H^{4,0}.$  In the later work \cite{HayashiNaumkin2013modified}, the regularity and localization requirements are relaxed, taking $W\in H^{1, \alpha+2\gamma},$ where $\alpha>\frac{1}{2}$ and $\gamma$ is small. 

Concerning the method of proof, both of these works apply a gauge transformation to convert the original quasilinear equation into a semilinear one, which then allows them to use a fixed point argument.

In contrast, our approach treats the problem as truly quasilinear, analyzing the equation directly, without transformations. 

This increases the difficulty but allows us to work with the equation in its natural form. To the best of our knowledge, this is the first proof of asymptotic completeness for a quasilinear dispersive equation. The cost of this more direct approach is a slight increase in the required regularity compared to \cite{HayashiNaumkin2013modified}.  The assumptions in our theorem are equivalent to taking  $W\in H^{\frac{3}{2},1}\cap H^{1,2}.$
We believe the techniques developed here can be adapted to other quasilinear dispersive problems.

In summary, while modified scattering and asymptotic completeness for DNLS have been studied before, our contribution lies in proving asymptotic completeness directly in the quasilinear setting, without relying on transformations. 

\subsection{Acknowledgements}
The author would like to thank her advisor, Mihaela Ifrim, for her valuable guidance and many helpful discussions.

\section{Preliminaries}
Here we will give some frequently used definitions as well as basic lemmas that will be used throughout the paper. 

\begin{definition}
    Define the Fourier Transform
    \begin{align*}
    \hat{f}(\xi):=\frac{1}{\sqrt{2\pi}}\int_{\R}e^{-ix\xi}f(x)\, dx.
\end{align*}
\end{definition}

\begin{definition}\label{LP defn}
    Let $P_\lambda $ denote the Littlewood Paley operator at frequency $\lambda$, i.e., for $\varphi$ a bump function supported on $[-2,2]$ and equal to $1$ on $[-1,1],$ define
\begin{align*}
  \widehat{P_{\leq \lambda}f}(\xi)=&\varphi\lp \frac{\xi}{\lambda}\rp \wh{f}(\xi),\\
    P_{>\lambda}:=&1-P_{\leq \lambda}.
\end{align*}
Let us point out that we will sometimes write $f_{\leq\lambda}$ for $P_\lambda f.$
Also define 
\begin{align*}
    P_{\lambda}:=P_{\leq\lambda}-P_{\leq\frac{\lambda}{2}}
\end{align*}
to be the Littlewood-Paley operator exactly localized to frequency $\lambda.$ 
\end{definition}
The next lemma is a nonlinear version of Gr\"onwall's inequality, which will be used with the energy estimates.
\begin{lemma}[Nonlinear Gr\"onwall]\label{NL Gronwall}
    For some functions $f,p,$ and $B,$ if 
    \begin{align*}
        \frac{d}{dt}f(t)^2\leq p(t) f(t)^2+ B(t)f(t),
    \end{align*}
    then
    \begin{align*}
          f(t)\leq& f(t_0) \cdot \exp\lp\int_{t_0}^t \frac{p(s)}{2}\, ds\rp+\frac{1}{2}\int_{t_0}^t B(s) \exp\lp\int_{s}^t -\frac{p(r)}{2}\, dr\rp \, ds.
    \end{align*}
    In particular, if $p(t)=C_1t^{-1}$ and $B(t)=C_2t^{-1-\beta},$ then
    \begin{align*}
             f(t)\leq& f(t_0) \cdot \lp \frac{t}{t_0}\rp^{C_\frac{1}{2}}+ \frac{C_2}{-2\beta+C_1}\lp t^{-\beta}-t_0^{-\beta+C_\frac{1}{2}}t^{-C_\frac{1}{2}} \rp.
    \end{align*}
\end{lemma}

\begin{proof}
The original equation is equivalent to 
    \begin{align*}
    f'\leq \frac{p(t)}{2}f+\frac{B(t)}{2}.
\end{align*}
We solve this as a linear ODE. Let $\mu$ be the integrating factor
\begin{align*}
    \mu(t)=\exp\lp\int_{t_0}^t -\frac{p(s)}{2}\, ds\rp.
\end{align*}
Then we get
\begin{align*}
    (f \cdot \mu(t))'\leq& \frac{B(t) \mu(t)}{2}\\
      f(t) \cdot \mu(t)\leq& f(t_0) \cdot \mu(t_0)+\int_{t_0}^t \frac{B(s) \mu(s)}{2} \, ds\\
            f(t)\leq& f(t_0) \cdot \exp\lp\int_{t_0}^t \frac{p(s)}{2}\, ds\rp+\exp\lp\int_{t_0}^t \frac{p(s)}{2}\, ds\rp\int_{t_0}^t \frac{B(s) \exp\lp\int_{t_0}^s \frac{p(r)}{2}\, dr\rp}{2} \, ds\\
     f(t)\leq& f(t_0) \cdot \exp\lp\int_{t_0}^t \frac{p(s)}{2}\, ds\rp+\frac{1}{2}\int_{t_0}^t B(s) \exp\lp\int_{s}^t -\frac{p(r)}{2}\, dr\rp \, ds.
\end{align*}
In the more specific case,
\begin{align*}
    \mu(t)= e^{-\frac{C_1}{2}(\log t-\log t_0)}= \lp \frac{t}{t_0}\rp^{-C_\frac{1}{2}}.
\end{align*}
So then the inequality is 
\begin{align*}
         f(t)\leq& f(t_0) \cdot \lp \frac{t}{t_0}\rp^{C_\frac{1}{2}}+\frac{1}{2}\int_{t_0}^t C_2s^{-1-\beta} \lp \frac{t}{s}\rp^{-C_\frac{1}{2}}\, ds\\
        f(t)\leq& f(t_0) \cdot \lp \frac{t}{t_0}\rp^{C_\frac{1}{2}}+\frac{C_2}{-2\beta+C_1}\lp t^{-\beta}-t_0^{-\beta+C_\frac{1}{2}}t^{-C_\frac{1}{2}} \rp.
\end{align*}
\end{proof}

\subsection{The linear equation \label{linear section}}
Before beginning the nonlinear analysis, we recall the analysis for the linear equation, which will provide some insight into what we want to look for in the nonlinear case. 
Here we work with the linear Schr\"odinger equation, which is given by 
\begin{equation}\begin{cases}
    iu_t+u_{xx}=0\\
   u(0,x)=u_0(x).
\end{cases}
\label{Lineareq}
\end{equation}
To solve the linear equation, we first take the Fourier transform to turn the PDE into a linear ODE:
\begin{align*}
    \hat{u_t}(\xi)+i\xi^2\hat{u}(\xi)=0,
\end{align*}
which has the solution 
\begin{align*}
    \hat{u}(t,\xi)=\hat{u_0}(\xi)e^{-i\xi^2t}.
\end{align*}
Therefore, we get the fundamental solution
\begin{align}\label{linear solution}
    u(t,x) \approx&u_0*t^{-\frac{1}{2}} e^{-i\frac{x^2}{4t}}.
\end{align} 
We can compare this with \eqref{u mod scat} to see that we do not obtain a classical scattering in the nonlinear case. 

It is also worth noting the role the vector field $L$ plays in the linear equation.  We define $L$ to be the pushforward of $x$ along the linear Schr\"odinger flow, $e^{it\px^2}Lu=x e^{it\px^2}u$, which is given by
\begin{equation}
\label{def:L}
    L:=x+2it\px .
\end{equation}
Thus, it commutes with the linear flow. A few key properties of $L$ are as follows:
\begin{lemma}
    \label{properties of L}
     \begin{align*}
 \left[ i\partial_t+\partial_x^2, L\right] =0, \quad    \left[ \partial_x , L\right]=1.
\end{align*}
We also have
\begin{align*}
    L(u\bv w)=Lu \bv w-u \widebar{Lv}w+uv Lw, 
\end{align*}
and $L$ is self-adjoint, $L^*=L.$
\end{lemma}
From the first commutator, we see that if $u$ is a solution to the linear equation, so is $Lu,$ and we have the conservation of mass 
\begin{align*}
    \|Lu(t)\|_{L^2}= \|Lu(0)\|_{L^2}.
\end{align*}
These properties will be useful when proving energy estimates for $L$ applied to the nonlinear equation, in section \ref{L EE}.

This vector field will be used extensively throughout this paper, as a tool to get bounds on $u.$  We will therefore review how it can be used in the linear case, as a more simple example. 
\begin{prop} For solutions $u$ to \eqref{Lineareq} with $\|u_0\|_{L^2}+\|xu_0\|_{L^2}\leq \epsilon$, we have the following dispersive bound:
    \begin{equation}
    \|u(t,x)\|_{L^\infty_x}\lesssim \epsilon t^{-\frac{1}{2}}.
\end{equation}
\end{prop}
\begin{proof} This can be proven directly by applying Young's convolution inequality to \eqref{linear solution}.  However, we will give a proof using vector fields that will carry over to the nonlinear case. From conservation of mass,
\begin{equation}
\label{L-estimate}
\Vert  u(t)\Vert_{L^2} + \Vert L u(t)\Vert_{L^2} \leq \epsilon.
\end{equation}
The main idea is to write $\frac{d}{dx}|u|^2$ in terms of the $L^2_x$ norms of $u$ and $Lu$, which we can control.  By applying $L$ to $u$ and solving for $u_x,$ we get
 \begin{align*}
u_x= \frac{1}{2t i}[Lu-xu].
\end{align*}
Then, we obtain 
\begin{align*}
    \frac{d}{dx}|u|^2
    =u\bu_x+\bu u_x=\frac{1}{2t i}[\bu Lu -u\widebar{Lu}].
\end{align*}
Gagliardo-Nirenberg gives
\begin{align*}
 \|u\|_{L^\infty}^2  &\lesssim\frac{1}{t}\int_{\R} |u||Lu| \, dx \lesssim \frac{1}{t}\|Lu\|_{L^2}\|u\|_{L^2}.
\end{align*}
Combining this with the energy bounds \eqref{L-estimate} we obtain the pointwise bound 
\begin{equation}
\|u\|_{L^\infty}\lesssim \epsilon t^{-\frac{1}{2}}.
\end{equation}
\end{proof}
We emphasize that one key step in this proof was establishing 
\begin{equation}
\|u(t)\|_{L^\infty}^2 \lesssim
\frac{1}{t} \|u(t)\|_{L^2} \|Lu(t)\|_{L^2}.
\end{equation}
This is called a Klainerman-Sobolev type inequality, and it is useful because it allows us to get the $L^\infty$ norm of the solution in terms of its energy and the vector field energy.  We will rely on this type of argument when proving the asymptotic completeness.

\subsection{Wave Packets and summary of previous work}
As this work is a continuation of previous work, we will need to review key ideas used in \cite{byars2024globaldynamicssmalldata}.
In this section we will review the notation and main definitions that are used throughout the paper.  We encourage the reader to look in \cite{byars2024globaldynamicssmalldata} for more motivation behind these definitions. 

The main technique needed is the method of \emph{testing by wave packets}, which was introduced by Ifrim and Tataru in 2014 for the NLS equation \cite{cite:NLS}.  Since then, it has often been used as an efficient way to obtain low regularity well-posedness. Some examples of this can be seen in \cite{aiifrimtataru2022two,albert.ovidiu2023gsqg,albert.ovidiu2024sqg,ifrimtatarucapillary,ifrimtataruholomorphic}, and the idea is explained in general in \cite{cite:WavePax}.

Recall that the dispersion relation for \eqref{dnls} is $a(\xi)=-\xi^2$, with group velocity $v=-2\xi.$  We will consider, for each $v$, wave packets traveling along the ray $x=vt$, which corresponds to the frequency $\xi_v = - v/2$.   
We consider wave packets of the form
\begin{align}
   \Phi_v= e^{i\phi (t,x)}\chi \left(\frac{x-vt}{\sqrt{t}}\right),
   \label{O(1/t)eqn}
\end{align}
where $\chi$ is a Schwartz function with 
\[
\int \chi(y)\, dy=1,
\]
and the phase $\phi(t,x)=\frac{x^2}{4t}$ is 
the same as the phase of the fundamental solution for the linear Schr\"odinger flow. 
In particular this guarantees that 
\begin{align*}
   ( i\pt+\px^2)\Phi_v=O\left(\frac{1}{t}\right).
\end{align*}

In order to properly compare our solution $u$ to the wave packets, we will also need to define an asymptotic profile $\gamma$ which is the $L^2$ inner product of our solution with the wave packet. 
\begin{align*}
    \gamma(t,v) :=\int u(t,x) \bar{\Phi}_v(t,x) \, dx.
\end{align*}

We will also need a few lemmas and equations from \cite{byars2024globaldynamicssmalldata}.

\begin{lemma}[Lemma 3.1 from \cite{byars2024globaldynamicssmalldata}]\label{gamma bnds lemma}
For $u$ a solution to \eqref{dnls}, and $\gamma$ as defined above, we have
    \begin{equation}
     \|\gamma\|_{L^\infty}\lesssim t^{\frac{1}{2}}\|u\|_{L^\infty}, \,\,\,\,\|\gamma\|_{L^2_v}\lesssim \|u\|_{L^2_x}, \,\,\,\, \|\partial_v \gamma \|_{L^2_v}\lesssim \|Lu\|_{L^2_x}.
        \label{gammabnds}
    \end{equation}
    Also,
    \begin{align}\label{v gamma L^infty}
  \|v\gamma\|_{L^\infty}\lesssim t^{-\frac{3}{4}}\|Lu\|_{L^2_x}+ t^{\frac{1}{2}}\|u_x\|_{L^\infty}  ,  \,\,\,\,\ \|\langle v \rangle^{\frac{k}{2}}\gamma\|_{L^\infty}\lesssim (\|Lu\|_{L^2_x} + \|u\|_{H_x^{k}}).
\end{align}
 We have spatial difference bounds
\begin{align}
 \|u(t,vt)-t^{-\frac{1}{2}}e^{i\phi(t,vt)}\gamma(t,v)\|_{L^\infty}\lesssim& t^{-\frac{3}{4}}\|Lu\|_{L^2_x}
      \label{spatialbnds1}\\
    \|u(t,vt)-t^{-\frac{1}{2}}e^{i\phi(t,vt)}\gamma(t,v)\|_{L^2_v}\lesssim& t^{-1}\|Lu\|_{L^2_x},
    \label{spatialL2}
\end{align}
and Fourier difference bounds
\begin{align}
  \|\hat{u}(t,\xi)-e^{-it\xi^2}\gamma(t,2\xi)\|_{L^\infty}  \lesssim
  & t^{-\frac{1}{4}}\|Lu\|_{L^2_x}
  \label{fourierbnds1}\\
    \|\hat{u}(t,\xi)-e^{-it\xi^2}\gamma(t,2\xi)\|_{L^2_\xi}\lesssim
    & t^{-\frac{1}{2}}\|Lu\|_{L^2_x}.
    \label{fourierbnds2}
\end{align}
\end{lemma}

\begin{lemma}[Lemma 3.2 from \cite{byars2024globaldynamicssmalldata}]
    Let $u$ be a solution to \eqref{dnls}. Then $\gamma$ satisfies
\begin{align}
\label{gamma asym}
    i\gamma_t(t)=\frac{v}{2}t^{-1}(|\gamma|^2\gamma)-R(t,v),
\end{align}
where the remainder $R$ satisfies the uniform  bounds
\begin{equation}
\begin{aligned}
    \|R\|_{L^\infty}\lesssim & \epsilon t^{-\frac{5}{4}+\eta}.
    \end{aligned}
    \end{equation}
\end{lemma}
The notation $\eta:=C\epsilon^2$ will be used throughout the paper for convenience.

The rest of the paper is split up into two main sections, one for each part of the main theorem.  First, in section \ref{MS sec}, we prove a modified scattering, which can be done in a standard way using wave packets, similar to \cite{cite:NLS}.  The second, the asymptotic completeness, is the main focus.  Here, we solve the problem backwards from infinity, using energy estimates along with a bootstrap argument and the vector field method to prove the well-posedness. This will be done in section \ref{AC sec}.

\section{Modified Scattering \label{MS sec}}
Recall that the goal of this section is to prove the modified scattering result for \eqref{dnls}, as given in part (a) of Theorem \ref{mainthm}.
This will be done by solving the asymptotic ODE \eqref{gamma asym} we obtained for the asymptotic profile $\gamma$, bounding the remainder $R$ \eqref{gamma asym}, and then expressing  $u$ asymptotically in terms of $\gamma.$ 
We begin with the following  $L^2$ bound on the remainder $R$.

\subsection{\texorpdfstring{$L^2$ bound for $R$}{L2 bound for R} \label{L^2 R bound section}}
\begin{lemma} \label{RL2bound}
For $R$ as defined in \eqref{gamma asym},
    \begin{align*}
    \|R(t,v)\|_{L^2_v}\lesssim \epsilon t^{-\frac{3}{2}+\eta}.
\end{align*}
\end{lemma}

\begin{proof}

Recall from \cite{byars2024globaldynamicssmalldata} that we can write $R=R_1+R_2+R_3+R_4$, where
\begin{align*}
    R_1=&\int \frac{i}{4t^2}e^{-i\phi}Lu \left[-i(x-vt)\chi(\cdot)+2t^{\frac{1}{2}}\chi'(\cdot)\right]\, dx,\\
    R_2=&-\int i(|u|^2u)t^{-\frac{1}{2}}\bar{\Psi}_v,\\
    R_3=&-\int \frac{v}{2}u\bar{\Phi}_v[|u|^2-|u(t,vt)|^2]\, dx,\\
    R_4=&\frac{v}{2}\gamma [t^{-1}|\gamma|^2-|u(t,vt)|^2].
\end{align*}
We have
\begin{align*}
    R_1=\frac{1}{2t}\left[\left(it^{\frac{1}{2}}\chi'(t^{\frac{1}{2}}v)+tv\chi(t^{\frac{1}{2}}v)\right)*_v\partial_v w(t,vt)\right],
\end{align*}
where we define
\begin{align}\label{w defn}
    w(t,x):=e^{-i\phi}u(t,x).
\end{align}
We can get the $L^2_v$ bound for $R_1$ using Young's inequality:
\begin{align*}
    \|R_1\|_{L^2_v}\lesssim \frac{1}{t}\|it^{\frac{1}{2}}\chi'(t^{\frac{1}{2}}v)+tv\chi(t^{\frac{1}{2}}v)\|_{L^1}\|\partial_v w(t,vt)\|_{L^2_v}\approx t^{-1}\|\partial_v w(t,vt)\|_{L^2_v}=t^{-\frac{3}{2}}\|Lu\|_{L^2_x}.
\end{align*}
To get the $R_2$ bound we need to define $\Psi_v:=e^{i\phi}\chi'\left(\frac{x-vt}{\sqrt{t}}\right)$ and $\Tilde{\gamma}:=\int u \bar{\Psi}_v\, dx$. This $\Tilde{\gamma}$ is very similar to $\gamma,$ so we expect it to have similar properties.  For example, we can get the bound:

\begin{align}
    \|\Tilde{\gamma}\|_{L^2_v}\lesssim \|u\|_{L^2_x},
\end{align}
by writing
\begin{align*}
    \Tilde{\gamma}=&\int u e^{-i\phi}\widebar{\chi'\left(\frac{x-vt}{\sqrt{t}}\right)}=t^{\frac{1}{2}}w*t^{\frac{1}{2}}\chi'(vt^{\frac{1}{2}}),
\end{align*}
so we have
\begin{align*}
    \|\Tilde{\gamma}\|_{L^2_v}\lesssim& t^{\frac{1}{2}}\|w\|_{L^2_v}\|t^{\frac{1}{2}}\chi'(vt^{\frac{1}{2}})\|_{L^1}
    \lesssim\|u\|_{L^2_x}.
\end{align*} 
Now we will use this to get the $L^2$ bound on $R_2$.
\begin{align*}
    |R_2|\lesssim& \|u\|_{L^\infty}^2t^{-\frac{1}{2}}\left|\int u\bar{\Psi}_v\right|,\\
    \|R_2\|_{L^2_v}\lesssim& \|u\|_{L^\infty}^2t^{-\frac{1}{2}}\|\Tilde{\gamma}\|_{L^2_v}
    \lesssim \|u\|_{L^\infty}^2t^{-\frac{1}{2}}\|u\|_{L^2_x}.
\end{align*}
For $R_3$, we have 

\begin{align*}
       |R_3|\lesssim& \|u\|_{L^\infty}\int u\bar{\Phi}_v[vu-vu(t,vt)]\, dx.
\end{align*}
First consider
\begin{align*}
     &\int u \bar{\Phi}_v[vu-vu(t,vt)]\, dx\\
     =&t^{-1}\int u \overline{\chi \lp \frac{x-vt}{\sqrt{t}}\rp}  e^{-i\phi}[xu(t,x)-vt u(t,vt)]\, dx\\
     =& t^{-1}\int  u\overline{\chi \lp \frac{x-vt}{\sqrt{t}}\rp}  e^{-i\phi}[Lu(t,x)-2ti u_x(t,x)-L u(t,vt)+2tiu_x (t,vt)]\, dx\\
       =& \underbrace{t^{-1}\int  u\overline{\chi \lp \frac{x-vt}{\sqrt{t}}\rp}  e^{-i\phi}[Lu(t,x)-L u(t,vt)]\, dx}_{I}
          + \underbrace{2i\int  u\overline{\chi \lp \frac{x-vt}{\sqrt{t}}\rp}  e^{-i\phi}[u_x (t,vt)- u_x(t,x)]\, dx}_{II}.
\end{align*}
The first term can be bounded in $L^2_v$ by 
\begin{align*}
    \|I\|_{L^2_v}\lesssim t^{-1}\|\gamma\|_{L^\infty}\|Lu\|_{L^2_v}\lesssim t^{-\frac{1}{2}+C\ep^2}\|u\|_{L^\infty}. 
\end{align*}
For the second term, we proceed as for the $L^\infty$ bounds in \cite{byars2024globaldynamicssmalldata}, using $w_2=e^{-i\phi}u_x$, the change of variables $x=(z+v)t$, and the fundamental theorem of calculus to get

\begin{align*}
|II|\leq& \|u\|_{L^\infty} \int  \overline{\chi \lp \frac{x-vt}{\sqrt{t}}\rp} [w_2 (t,vt)- w_2(t,x)]\, dx\\
   \leq&\|u\|_{L^\infty} \int \overline{\chi}(t^{\frac{1}{2}}z) [w_2 (t,vt)- w_2(t,t(z+v))]\, dx\\
=&\|u\|_{L^\infty}  \int \overline{\chi}(t^{\frac{1}{2}}z) \int_0^1 z \pv w_2(t,t(v+hz))\, dh \, dz,\\
\|II\|_{L^2_v}\leq& \|u\|_{L^\infty}\|\pv w_2\|_{L^2_v} \int |z| \overline{\chi}(t^{\frac{1}{2}}z) \, dz= t^{-1}\|u\|_{L^\infty}\|\pv w_2\|_{L^2_v} \lesssim t^{-\frac{3}{2}}\|u\|_{L^\infty}\|Lu_x\|_{L^2_x}.
\end{align*}
Then we have
\begin{align*}
    \|R_3\|_{L^2_v}\lesssim t^{-\frac{1}{2}+C\ep^2}\|u\|_{L^\infty}^2+t^{-\frac{3}{2}}\|u\|_{L^\infty}^2\|Lu_x\|_{L^2_x}\lesssim \ep t^{-\frac{3}{2}+\eta}.
\end{align*}
For $R_4$, we can use the same strategy as for the $L^\infty$ bound to get
\begin{align}
     \|R_4\|_{L^2_v}=&\left\|\frac{v}{2}\gamma [t^{-1}|\gamma|^2-|u(t,vt)|^2]\right\|_{L^2_v}\\
     \lesssim&\|v\gamma\|_{L^\infty}\|u\|_{L^\infty}\|u(t,vt)-t^{-\frac{1}{2}}e^{i\phi}\gamma\|_{L^2_v}\\
     \lesssim &\|v\gamma\|_{L^\infty}\|u\|_{L^\infty}t^{-1}\|Lu\|_{L^2}\lesssim \epsilon \langle t\rangle^{-\frac{3}{2}}(t^{-\frac{3}{4}}\|Lu\|_{L^2_x}+ t^{\frac{1}{2}}\|u_x\|_{L^\infty} )\|Lu\|_{L^2_x}.
\end{align}
Now overall we have
\begin{align*}
    \|R(t,v)\|_{L^2_v}\lesssim \epsilon t^{-\frac{3}{2}+\eta},
\end{align*}
which proves the Lemma. 
\end{proof}
Next, we solve the ODE \eqref{gamma asym} for $\gamma.$
\begin{lemma}
\label{gammalemma}
    We can write $\gamma$ as
    \begin{align}\label{gamma}
 \gamma(t,v)=W(v)e^{\frac{-iv}{2}|W(v)|^2\log t}+err_1,
\end{align}
where $err_1\in O_{L^\infty_v}(\epsilon t^{-\frac{1}{4}+\eta})\cap O_{L^2}(\epsilon t^{-\frac{1}{2}+\eta}).$
\end{lemma}
\begin{proof}
We can solve the equation
\begin{align*}
    i\gamma_t(t)= \, t^{-1}\frac{v}{2}(|\gamma|^2\gamma)-R(t,v),
\end{align*} 
by first solving the principal part
\begin{align*}
    i\gamma^*_t(t)= \, t^{-1}\frac{v}{2}(|\gamma^*|^2\gamma^*),
\end{align*} 
to get
\begin{align*}
    \gamma^*(t,v)=W(v)e^{-\frac{iv}{2}|W(v)|^2\log t},
\end{align*}
for some $W(v)$. We get the $L^2$ bound
\begin{align}\label{gamma2}
 \gamma(t,v)=W(v)e^{-\frac{iv}{2}|W(v)|^2\log t}+O_{L^2}(\epsilon t^{-\frac{1}{2}+\eta}),
\end{align}
in the following way. Define $h:=\gamma-\gamma^*.$
We can get the bound
 $\|h\|_{L^2}\lesssim \epsilon t^{-\frac{1}{2}+\eta}$, by doing an energy estimate for $h.$ Notice that $h$ solves the equation:
 \begin{align}
    ih_t=t^{-1}\frac{v}{2}[|\gamma|^2\gamma-|\gamma^*|^2\gamma^*]-R(t,v).\label{h eqn}
\end{align}
 Now define $f:=|\gamma|^2\gamma-|\gamma^*|^2\gamma^*$
and apply the following algebraic manipulation:
\begin{align*}
   |\gamma|^2\gamma-|\gamma^*|^2\gamma^*
   =&|\gamma|^2[\gamma-\gamma^*]+\gamma^*[\gamma\bar{\gamma}-\gamma^*\bar{\gamma^*}]\\
   =&|\gamma|^2[\gamma-\gamma^*]+\gamma^*[\bar{\gamma}(\gamma-\gamma^*)-\gamma^*(\bar{\gamma^*}-\bar{\gamma})]\\
    =&|\gamma|^2h+\gamma^*[\bar{\gamma}h+\gamma^*\bar{h}].
\end{align*}
By doing the usual energy estimates,
we get
\begin{align*}
    \frac{d}{dt}\|h\|_{L^2}^2\lesssim \left|\int t^{-1}v\Rea(i[f]\bar{h})\right|+\left|\int R(t,v)\bar{h}\right|.
\end{align*}
The second term can be bounded via H\"older's inequality by
\begin{align*}
    \|R\|_{L^2_v}\|h\|_{L^2_v}. 
\end{align*}
For the first term, we can use the algebraic manipulation above to write
\begin{align*}
    \Rea (if\bar{h})=\Rea(i\gamma^*[\bar{\gamma}|h|^2+\gamma^*(\bar{h})^2]).
\end{align*}
Integrating this, we obtain
\begin{align*}
     \left|\int t^{-1}v\Rea(i[f]\bar{h})\right|=&t^{-1}\int |v \gamma^*\gamma h^2|+|v|\gamma^*|^2|h|^2\\
     \lesssim& t^{-1}\left[\|\gamma^*\|_{L^\infty}\|v\gamma\|_{L^2_v}+\|\gamma^*\|_{L^\infty}\|v\gamma^*\|_{L^2_v}\right]\|h\|_{L^2}^2\\
     \lesssim& \epsilon t^{-1}\|h\|_{L^2_v}^2.
\end{align*}
Note that
\begin{align*}
    \|\gamma^*\|_{L^2}=&\|W(v)\|_{L^2},\\
    \|\gamma^*\|_{L^\infty}=&\|W(v)\|_{L^\infty}.
\end{align*}
These are both conserved in time and only depend on the initial data. Overall we have 
\begin{align}
     \frac{d}{dt}\|h\|_{L^2}^2
     \lesssim&\epsilon t^{-1}\|h\|_{L^2_v}^2+\ep t^{-\frac{3}{2}}\|h\|_{L^2_v}.
\end{align}
By applying Nonlinear Gr\"onwall \ref{NL Gronwall}, we get
\begin{align*}
    \|h\|_{L^2_v}\leq \ep t^{-\frac{1}{2}}.
\end{align*}
Similarly, we need to prove an $L^\infty$ bound on $h.$  From \eqref{h eqn}, we have
\begin{align*}
    \frac{d}{dt}\|h\|_{L^\infty}= t^{-1}\lp\|\gamma\|_{L^\infty}\|v\gamma\|_{L^\infty}+\|\gamma^*\|_{L^\infty} \|v\gamma\|_{L^\infty}+\|\gamma^*\|_{L^\infty}\|v\gamma^*\|_{L^\infty}\rp \|h\|_{L^\infty}+\|R(t,v)\|_{L^\infty}.
\end{align*}
Then solving this linear ODE and using that $\|R\|_{L^\infty}\lesssim \ep t^{-\frac{5}{4}+\eta}$, we get
\begin{align*}
    \|h\|_{L^\infty}\lesssim \ep t^{-\frac{1}{4}+\eta}.
\end{align*}
Now we have
\begin{align}\label{gamma in terms of W}
 \gamma(t,v)=W(v)e^{-\frac{iv}{2}|W(v)|^2\log t}+err_1,
\end{align}
where $err_1\in O_{L^\infty_v}(\epsilon t^{-\frac{1}{4}+\eta})\cap O_{L^2}(\epsilon t^{-\frac{1}{2}+\eta}).$
\end{proof}

There are now two steps left in order to complete the modified scattering: use what we have to write $u,$ and find the regularity of $W.$  This can be summarized in the following lemma:
\begin{lemma}  \label{u in terms of W lemma} 
The solution $u$ can be asymptotically expanded as
\begin{align*}
    u(t,x)=t^{-\frac{1}{2}}e^{i\frac{x^2}{4t}}W\left(\frac{x}{t}\right)e^{-\frac{ix}{2t}|W\left(\frac{x}{t}\right)|^2\log t}+err_x,
\end{align*}
for some $W(v)\in H^{1-C\epsilon^2,1-C\epsilon^2}.$
\end{lemma}

\subsection{\texorpdfstring{Solving for $u$}{Solving for u}}

Using the difference bounds  \eqref{spatialbnds1}, \eqref{spatialL2}, we can write
\begin{align*}
    u(t,x)=t^{-\frac{1}{2}}e^{i\phi(t,x)}\gamma\left(t,\frac{x}{t}\right)+err_x,
\end{align*}
where
\begin{align*}
    \|err_x\|_{L^\infty}\lesssim& t^{-\frac{3}{4}}\|Lu\|_{L^2_x}\lesssim \epsilon (1+t)^{-\frac{3}{4}+\eta},\\
      \|err_x\|_{L^2}\lesssim& t^{-1}\|Lu\|_{L^2_x}\lesssim \epsilon (1+t)^{-1+\eta}.
\end{align*}
Then, using (\ref{gamma}), we get
\begin{align*}
    u(t,x)
    =&t^{-\frac{1}{2}}e^{i\frac{x^2}{4t}}W\left(\frac{x}{t}\right)e^{-\frac{ix}{2t}|W\left(\frac{x}{t}\right)|^2\log t}+err_x,
\end{align*}
where now this newly denoted $err_x$ is also bounded as
\begin{align*}
    \|err_x\|_{L^\infty}\lesssim& \epsilon (1+t)^{-\frac{3}{4}+\eta},\\
      \|err_x\|_{L^2}\lesssim& \epsilon (1+t)^{-1+\eta}.
\end{align*}
In order to get (\ref{hatu mod scat}), we use the Fourier bounds \eqref{fourierbnds1}, \eqref{fourierbnds2} to get
\begin{align*}
    \hat{u}(t,\xi)=e^{it\xi^2}\gamma(t,2\xi)+err_\xi,
\end{align*}
where 
\begin{align}
    \|err_\xi\|_{L^\infty}\lesssim t^{-\frac{1}{4}}\|Lu\|_{L^2_x}\lesssim \epsilon (1+t)^{-\frac{1}{4}+\eta},\label{errxibnd}\\
      \|err_\xi\|_{L^2}\lesssim t^{-\frac{1}{2}}\|Lu\|_{L^2_x}\lesssim \epsilon (1+t)^{-\frac{1}{2}+\eta}.
\end{align}
Then using (\ref{gamma}), we get
\begin{align*}
    \hat{u}(t,\xi)=&e^{it\xi^2}[W(2\xi)e^{-i\xi|W(2\xi)|^2\log t}+err_1]+err_\xi\\
    =&e^{it\xi^2}W(2\xi)e^{-i\xi|W(2\xi)|^2\log t}+err_\xi,
\end{align*}
where this newly denoted $err_\xi$ still obeys (\ref{errxibnd}).

\subsection{\texorpdfstring{Regularity of $W$}{Regularity of W}}
Recall $w$ is defined as in \eqref{w defn},
and notice that for $v=x/t,$ we have
\begin{align*}
    \|u(0)\|_{L^2_x}=  \|u(t)\|_{L^2_x}=  t^{\frac{1}{2}}\|w(t)\|_{L^2_v}.
\end{align*}
By \eqref{gamma in terms of W},
\begin{align*}
      \|W\|_{L^2_v}=  \|\gamma\|_{L^2_v}+ err_x.
\end{align*}
Rearranging the difference bound \eqref{spatialL2}, we have
\begin{align*}
    \|t^{\frac{1}{2}}e^{-i\phi(t,vt)}u(t,vt)-\gamma(t,v)\|_{L^2_v}\lesssim& t^{-\frac{1}{2}}\|Lu\|_{L^2_x},
\end{align*}
which implies that
\begin{align*}
\|\gamma\|_{L^2_v}=t^{\frac{1}{2}}\|w(t)\|_{L^2_v}+err_x.
\end{align*}
Therefore we have
\begin{align*}
    \|W\|_{L^2_v}=\|u\|_{L^2_x}+ err_x,
\end{align*}
where $err_x= O_{L^2}(\epsilon t^{-\frac{1}{2}+\eta})$.

Now, to get the regularity of $W,$ we proceed with an interpolation argument.  The idea is to break it up into two pieces, estimate each piece separately, and then localize in frequency and space (on the uncertainty principle scale), and interpolate to get the final bound. We write 
\begin{align*}
    W= \underbrace{W(v)-\gamma(t,v)e^{i\frac{v}{2}|\gamma(t,v)|^2\log t}}_{f}+\underbrace{\gamma(t,v)e^{i\frac{v}{2}|\gamma(t,v)|^2\log t}}_g.
\end{align*}
First, estimate $f$ in $L^2.$ 
By solving \eqref{gamma2} for $W$ and plugging it in, we get
\begin{align*}
 \|W(v)-\gamma(t,v)e^{i\frac{v}{2}|\gamma(t,v)|^2\log t}\|_{L^2_v}=& \|\gamma(t,v)e^{i\frac{v}{2}|W(v)|^2\log t}-\gamma(t,v)e^{i\frac{v}{2}|\gamma(t,v)|^2\log t}\|_{L^2_v}+O(\epsilon t^{-\frac{1}{2}+\eta})\\
 =&\left\|\gamma(t,v)\left[e^{-i\frac{v}{2}|W(v)|^2\log t}-e^{-i\frac{v}{2}|\gamma(t,v)|^2\log t}\right]\right\|_{L^2_v}+O(\epsilon t^{-\frac{1}{2}+\eta})\\
 \lesssim& \|\gamma\|_{L^\infty}\left\|{v|W(v)|^2\log t}-{v|\gamma(t,v)|^2\log t}\right\|_{L^2_v}+O(\epsilon t^{-\frac{1}{2}+\eta})\\
  \lesssim& \|v\gamma\|_{L^\infty}\log t\left\|{|W(v)|^2}-{|\gamma(t,v)|^2}\right\|_{L^2_v}+O(\epsilon t^{-\frac{1}{2}+\eta})\\
  \lesssim& \ep \log t\left\|{W(v)\bar{W}(v)}-\gamma(t,v)\bar{W}+\gamma(t,v)\bar{W}-\gamma(t,v)\bar{\gamma}(t,v)\right\|_{L^2_v}\\
  &+O(\epsilon t^{-\frac{1}{2}+\eta})\\
  \lesssim& \ep \log t[\|W\|_{L^\infty}+\|\gamma\|_{L^\infty}]\left\|{W(v)}-\gamma(t,v)\right\|_{L^2_v}+O(\epsilon t^{-\frac{1}{2}+\eta})\\
  \lesssim&\epsilon t^{-\frac{1}{2}+\eta},
\end{align*}
where the third line follows from the Lipschitz continuity of $e^{ix}$ when $x\in \R$, and the last follows from $\gamma$ and $W$ being bounded in $L^\infty.$
Next, we want to bound $g$ in $H^{1,1}$. To do this, calculate
\begin{align*}
\pv(v g)=&    \partial_v[v\gamma(t,v)e^{i\frac{v}{2}|\gamma|^2\log t}]\\
=& \gamma(t,v)e^{i\frac{v}{2}|\gamma|^2\log t}+v\partial_v\gamma(t,v)e^{i\frac{v}{2}|\gamma|^2\log t}+ i\frac{v^2}{2}\gamma(t,v)e^{i\frac{v}{2}|\gamma|^2\log t}\cdot \log t(\gamma\bar{\gamma}_v+\bar{\gamma}\gamma_v)\\
    +&\frac{v}{2}\gamma(t,v)e^{i\frac{v}{2}|\gamma|^2\log t}\cdot \log t(|\gamma|^2),\\
    \| \partial_v(vg)\|_{L^2_v}\lesssim&\|\gamma(t,v)\|_{L^2_v}+\|v\partial_v\gamma(t,v)\|_{L^2_v}+\log t \|v\gamma\|_{L^\infty}^2\|\partial_v\gamma(t,v)\|_{L^2_v}\\
    +&\log t \|v\gamma\|_{L^\infty}\|\gamma\|_{L^\infty}\|\gamma(t,v)\|_{L^2_v}\\
    \lesssim& \ep+ \|v\partial_v\gamma(t,v)\|_{L^2_v}+ \ep^3  t^{\eta}\log t.
\end{align*}
All the terms can be bounded by Lemma \ref{gamma bnds lemma}, except the second one. It remains to check $v\pv \gamma$. 
From \cite{byars2024globaldynamicssmalldata},
we have
\begin{align*}
\pv \gamma=t^{\frac{1}{2}}\pv w(t,vt)* t^{\frac{1}{2}}\chi(vt^{\frac{1}{2}}),
\end{align*}
where $w$ is as in \eqref{w defn}.
By the convolution property of multiplication, we split this term further,
\begin{align*}
    v \pv \gamma=\underbrace{t^{\frac{1}{2}}v\pv w(t,vt)* t^{\frac{1}{2}}\chi(vt^{\frac{1}{2}})}_I+\underbrace{t^{\frac{1}{2}}\pv w(t,vt)* vt^{\frac{1}{2}}\chi(vt^{\frac{1}{2}})}_{II}.
\end{align*}
The first term is
\begin{align*}
   I= &\int t^{\frac{1}{2}}y \p_y w(t,yt) \cd t^{\frac{1}{2}}\chi((v-y)t^{\frac{1}{2}}) \,dy\\
    =&\int ty \lp\frac{-i}{2}e^{-i\phi(t,yt)}Lu(yt) \rp \chi((v-y)t^{\frac{1}{2}}) \,dy\\
    =&\int \p_y \lp e^{-i\phi(t,yt)} \rp Lu(yt)\chi((v-y)t^{\frac{1}{2}}) \,dy\\
     =-&\int \lp e^{-i\phi(t,yt)} \rp \lefb \p_y Lu(yt)\chi((v-y)t^{\frac{1}{2}})-t^{\frac{1}{2}}Lu(y)\chi'((v-y)t^{\frac{1}{2}}) \rigb \,dy.
\end{align*}
Then taking the $L^2_v$ norm and using Young's inequality,
\begin{align*}
     \|I\|_{L^2_v}\leq& t^{-\frac{1}{2}}\|\p_v Lu\|_{L^2_v}\|t^{\frac{1}{2}}\chi(vt^{\frac{1}{2}})\|_{L^1_v}+\|Lu\|_{L^2_v}\|t^{\frac{1}{2}}\chi'(vt^{\frac{1}{2}})\|_{L^1_v}\\
   =& t^{-1}\|\p_v Lu\|_{L^2_x}+t^{-\frac{1}{2}}\|Lu\|_{L^2_x}\\
   =& \|\p_x Lu\|_{L^2_x}+t^{-\frac{1}{2}}\|Lu\|_{L^2_x}\\
   \leq& \ep t^{\eta}.
\end{align*}
The second term can also be bounded as
\begin{align*}
    \|II\|_{L^2_v}=\|t^{\frac{1}{2}}\pv w\|_{L^2_v}\|vt^{\frac{1}{2}}\chi(vt^{\frac{1}{2}})\|_{L^1_v}\approx =\|\pv w\|_{L^2_x}=\|Lu\|_{L^2_x}\lesssim \ep t^\eta.
\end{align*}
This results in
\begin{align*}
g=O_{H^{1,1}}(\ep t^{\eta}\log t).
\end{align*}
Putting the bounds for $f$ and $g$ together, we have
\begin{align*}
    W(v)=O_{L^2_v}(\epsilon t^{-\frac{1}{2}+\eta})+O_{H^{\MI{1,1}}_v}(\log t \epsilon t^{\eta}).
\end{align*}
Now we will localize in space and frequency in order to interpolate and find the regularity of $W.$ We
denote the dyadic frequency, respectively spatial localization scales by $\lambda, r \gtrsim 1$, where 
the frequencies $|\xi|\lesssim 1$ and the spatial region $|v| \lesssim 1$ are grouped together as a single piece, since we are working in inhomogeneous spaces. For the localizations we use the operators $P_\lambda$ in frequency, respectively $T_r$ in position. We estimate the (doubly) dyadic pieces of $g$ by 
\[
\|P_\lambda T_r g\|_{L^2} \lesssim \frac{1}{\lambda r}
\|g\|_{H^{1,1}}, \qquad r,\lambda \geq 1.
\]
Then
\begin{align*}
      \|P_\lambda T_r W\|_{L^2}\lesssim& \|P_\lambda T_r f\|_{L^2}+  \|P_\lambda T_r g\|_{L^2} \\
      \lesssim& \ep\log t \lp t^{\eta}\lambda^{-1}r^{-1}+ t^{-\frac{1}{2}+\eta}\rp.
\end{align*}

Here we optimize the choice of $t$ depending on both $\lambda$ and $r$. This is achieved when $t\sim (\lambda r)^{2}$ so we get
\begin{align*}
      \|P_\lambda T_r W\|_{L^2}\lesssim& \ep (\lambda r)^{-1+4\eta}.
\end{align*}
Since $\eta\approx \ep^2$, for $C$ large enough, after dyadic summation, this gives
\begin{align*}
       \| W\|_{H^{1-C\ep^2, 1-C\ep^2}}\lesssim& \ep^{-1} ,
\end{align*}
where the dyadic summation loses\footnote{This loss can be avoided by replacing the Sobolev exponent $1-C\epsilon^2$ by $1-c$ with a fixed positive constant $c$.} an $\epsilon^{-2}$ 
factor.
This completes the proof of Lemma \ref{u in terms of W lemma} and therefore proves the modified scattering.

\section{Asymptotic Completeness \label{AC sec}}

Now we will prove the asymptotic completeness of the problem, which is part (b) of Theorem \ref{mainthm}.  This result is given in the following proposition.
\begin{prop} \label{Prop AC}
    Given an asymptotic profile $W$ satisfying 
    \begin{equation}
    \label{Wcond}
    \|W\|_{H^{3/2+2\delta}_v}\leq M\ll \epsilon, \,\,\|vW\|_{H^{3/2+\delta}_v}\leq M, \,\,\|v^2W\|_{H^{1}_v}\leq M,\,
    \end{equation}
    with $\delta, M>0, \,\,M^2\ll \delta \ll 1$, then there exists a unique initial data $u_0$ with
    \begin{align}\label{FPinitials}
       \|u_0\|_{H^1}\leq \epsilon,\,\,\,
           \|xu_0\|_{H^1}\leq \epsilon,
    \end{align}
    such that \eqref{u mod scat} and \eqref{hatu mod scat} hold.
\end{prop}

We start by defining an ansatz for our solution:
\begin{align*}
    u_{asymp}
=t^{-\frac{1}{2}}e^{i\frac{x^2}{4t}}e^{-\frac{ix}{2t}|W\left(\frac{x}{t}\right)|^2\log t}W\left(\frac{x}{t}\right).
\end{align*}
One would expect this to be a good approximate solution for  which we would like to find an exact solution matching it at infinity.
However, as it turns out,  if $W$ is at low regularity then this ansatz is not a good  enough approximate solution. We rectify this  by replacing $W$ with a time dependent frequency localized version. Let
$\mathcal{W}:=W_{\leq t^{\frac{1}{2}}}(v)$
be the Littlewood-Paley projection to frequencies $\leq t^{\frac{1}{2}}$, as in Definition \ref{LP defn}. Then we define a better ansatz as
\begin{align*}
    u_{app}:=t^{-\frac{1}{2}}e^{i\frac{x^2}{4t}}\W\left(\frac{x}{t}\right)e^{-\frac{ix}{2t}|\W\left(\frac{x}{t}\right)|^2\log t}.
\end{align*}

We need to compare $\uap$ with $u_{asymp}$ to make sure this  is an acceptable substitution. Before doing this, we will prove a key lemma.

\begin{lemma}\label{Wbounds_lemma}
Let $W$ be as in \eqref{Wcond}. Then for the frequency truncated $\W$ we have the following bounds:
    \begin{align}\label{W bounds row 1}
      \|\W\|_{L^\infty}\lesssim M, \,\,\,     \|\W'\|_{L^\infty}\leq M, \,\,\, \|\W'\|_{L^2}\leq M, \,\,\,  \|\W''\|_{L^2}\leq Mt^{\frac{1}{4}-\delta},   \\
  \|\W'''\|_{L^2}\leq Mt^{\frac{3}{4}-\delta}, \,\,\, \|\W''\|_{L^\infty}\leq Mt^{\frac{1}{2}-\frac{\delta}{2}} \,\,\,  \|v\W\|_{L^\infty}\lesssim M,\label{W bounds row 2}\\
  \,\, \,\|(v\W)'\|_{L^2}\leq M\,\,\,\|v\W'\|_{L^\infty}\leq M, \,\, \,\|v\W''\|_{L^2}\leq Mt^{\frac{1}{4}-\delta}, \,\,\, \|v^2\W'\|_{L^2}\leq M,
    \end{align}
    and the difference bounds
    \begin{align}
\label{differenceW3/2}
     \|W(v)-\W(v)\|_{L^2_v}\lesssim& Mt^{-\frac{3}{4}-\delta}, \,\,\,\,\,\, \|W(v)-\W(v)\|_{L^\infty}\lesssim Mt^{-\frac{1}{2}-\delta}.
\end{align}
We also have the bounds on the time derivatives:
\begin{align*}
\| \pt \W\|_{L^2_v}=\frac{1}{t}\|W_{t^{\frac{1}{2}}}\|_{L^2_v}\leq M t^{-\frac{3}{2}},\,\,\,\| \pt \W'\|_{L^2_v}\leq M t^{-\frac{5}{4}-\delta},\\
 \|\pt (v\W)\|_{L^2}\leq Mt^{-\frac{3}{2}} ,\,\,\,
\| \pt v\W'\|_{L^2_v}\leq M t^{-\frac{5}{4}-\delta},
\end{align*}
\end{lemma}
where $W_{t^{\frac{1}{2}}}$ is localized exactly to frequency $t^{\frac{1}{2}}.$

\begin{proof}
We will be using the Sobolev embedding inequality
\begin{align*}
    \|u\|_{L^\infty(\R^d)}\lesssim \|u\|_{H^k(\R^d)},
\end{align*}
which is true for $k>\frac{d}{2}$. Using this, and the frequency localization of $\W,$ we obtain
    \begin{align*}
    \|\W\|_{L^\infty}\leq& \|\W\|_{H^{3/2+2\delta}}\leq M,\\
    \|\W'\|_{L^\infty}\leq& \|\W'\|_{H^{1/2+2\delta}}= \|\xi^{\frac{3}{2}+2\delta}\wh{\W}\|_{L^2_\xi}=\|\W\|_{H^{3/2+2\delta}}\leq M,\\
     \|\W''\|_{L^\infty}\leq&\|\W''\|_{H^{1/2+\delta}}= \|\xi^{\frac{5}{2}+\delta}\wh{\W}\|_{L^2_\xi}\leq t^{\frac{1}{2}-\frac{\delta}{2}}\|\W\|_{H^{3/2+2\delta}}   \leq Mt^{\frac{1}{2}-\frac{\delta}{2}},\\
    \|\W'\|_{L^2}=&\|\W\|_{\dot{H}^2}\leq \|W\|_{H^{3/2+2\delta}}\leq M.
\end{align*}
For $v\W$, we use the spatial localization of $\W$ to get
\begin{align*}
      \|v\W\|_{L^\infty}\leq& \|v\W\|_{H^{3/2+2\delta}}\lesssim M.
\end{align*}
In general, we have
\begin{align*}
     \|v\W'\|_{L^p}\leq&  \|\W\|_{L^p} +  \|(v\W)'\|_{L^p},
\end{align*}
so it is enough to show
\begin{align*}
   \|(v\W)'\|_{L^2}=& \|\la \xi\ra \wh{v\W}\|_{L^2_\xi}\leq \|v\W\|_{H^{3/2+\delta}} \leq M, \\
    \|(v\W)'\|_{L^\infty} \lesssim& \|(v\W)'\|_{H^{1/2+2\delta}}\leq \|v\W\|_{H^{3/2+2\delta}}\lesssim M,\\
    \|(v\W)''\|_{L^2} =& \|\la \xi\ra^{2}\wh{v\W}\|_{L^2_\xi}= \|\la \xi\ra^{\frac{3}{2}+2\delta}\la \xi\ra^{\frac{1}{2}-2\delta}\wh{v\W}\|_{L^2_\xi}
    \lesssim t^{\frac{1}{4}-\delta}\|v\W\|_{H^{3/2+2\delta}}\leq Mt^{\frac{1}{4}-\delta}.
\end{align*}
For the difference bounds \eqref{differenceW3/2}, we first use triangle inequality, to get
\begin{align*}
    \|W-\W\|_{H^{3/2+2\delta}_v}\leq \|W\|_{H^{3/2+2\delta}_v}+\|\W\|_{H^{3/2+2\delta}_v}\lesssim M .
\end{align*}
Then, since $W-\W=W_{> t^{\frac{1}{2}}}$, we have
\begin{align*}
    \|(W(v)-\W(v))t^{\frac{3}{4}+\delta}\|_{L^2_v}=&\|(\widehat{W}(\xi)-\widehat{\W}(\xi))t^{\frac{3}{4}+\delta}\|_{L^2_\xi}\leq \|(\widehat{W}(\xi)-\widehat{\W}(\xi))\xi^{\frac{3}{2}+2\delta}\|_{L^2_\xi}\\
    \leq& \|W(v)-\W(v)\|_{H^{3/2+2\delta}_v}\lesssim M.
\end{align*}
This proves the $L^2$ difference bound. For the $L^\infty$ difference bound, we split up the difference using Littlewood-Paley:
\[
\Vert \W-W\Vert_{L^{\infty}} =\left\Vert W_{> t^{\frac{1}{2}}}\right\Vert_{L^{\infty}} = \left\Vert \sum_{\lambda > t^{\frac{1}{2}}}W_{\lambda}\right\Vert_{L^{\infty}}\leq \sum_{\lambda > t^{\frac{1}{2}}} \Vert W_{\lambda}\Vert_{L^{\infty}}.
\]
Then using Bernstein's inequality, and the fact that $M\geq\|W_\lb\|_{H^{3/2+2\delta}}=\|\xi^{\frac{3}{2}+2\delta}\widehat{W_\lb}\|_{L^2}\approx \lb^{\frac{3}{2}+2\delta}\|W_\lb\|_{L^2}$, we get
\begin{align*}
    \sum_{\lambda > t^{\frac{1}{2}}} \Vert W_{\lambda}\Vert_{L^{\infty}}\leq&
\sum_{\lambda >t^{\frac{1}{2}}} \lambda^{\frac{1}{2}} \Vert W_{\lambda}\Vert_{L^{2}}\leq  \sum_{\lambda > t^{\frac{1}{2}}} \lambda^{\frac{1}{2}}\cdot M\lambda^{-\frac{3}{2}-2\delta} =M\sum_{\lambda > t^{\frac{1}{2}}} \lambda^{-1 -2\delta } \\
\leq&M\sum_{n=1}^\infty (2^nt^{\frac{1}{2}})^{-1 -2\delta } \leq Mt^{-\frac{1}{2}-\delta}\sum_{n=1}^\infty (2^n)^{-1 -2\delta }
\lesssim Mt^{-\frac{1}{2}-\delta},
\end{align*}
which gives
\begin{align*}
        \|W(v)-\W(v)\|_{L^2_v}\lesssim& Mt^{-\frac{3}{4}-\delta}, \,\,\,\,\,\, \|W(v)-\W(v)\|_{L^\infty}\lesssim Mt^{-\frac{1}{2}-\delta}.
\end{align*}
For the time derivatives, we use the definition of $\W$ to get
\begin{align*}
    \pt  \wh{\W}=&-\frac{1}{2}\xi t^{-\frac{3}{2}}\varphi'\lp\frac{\xi}{t^{\frac{1}{2}}} \rp \wh{W}(\xi)=-\frac{1}{2}t^{-1}\cd\underbrace{ \frac{\xi}{t^{\frac{1}{2}}}\varphi'\lp\frac{\xi}{t^{\frac{1}{2}}} \rp \wh{W}(\xi)}_{\wh{W_{t^{\frac{1}{2}}}}},\\
    \pt \W=&t^{-1}W_{t^{\frac{1}{2}}}.
\end{align*}
Note that $W_{t^{\frac{1}{2}}}$ is $W$ localized exactly to the frequency $t^{\frac{1}{2}}$. 
To get the $L^2_v$ bound,
\begin{align*}
    \|W_{t^{\frac{1}{2}}}\|_{L^2_v}=&\|\wh{W_{t^{\frac{1}{2}}}}\|_{L^2_\xi}=\left\| \frac{\xi}{t^{\frac{1}{2}}}\varphi'\lp\frac{\xi}{t^{\frac{1}{2}}} \rp \wh{W}(\xi)  \right\|_{L^2_\xi}=t^{-\frac{1}{2}}\left\| \xi\varphi'\lp\frac{\xi}{t^{\frac{1}{2}}} \rp \wh{W}(\xi)  \right\|_{L^2_\xi}\\
    \lesssim& t^{-\frac{1}{2}} \|\W'\|_{L^2_v}\lesssim Mt^{-\frac{1}{2}}.
\end{align*}
For $\pt \W'$, we have
\begin{align*}
\widehat{\pv \W}(\xi)=& \xi \wh{\W}(\xi)=\xi \varphi\lp\frac{\xi}{t^{\frac{1}{2}}} \rp \wh{W}(\xi),\\
\pt \widehat{\pv \W}(\xi)=&-\frac{1}{2}\xi^2 t^{-\frac{3}{2}}\varphi'\lp\frac{\xi}{t^{\frac{1}{2}}} \rp \wh{W}(\xi),\\
\| \pt \W'\|_{L^2_v}=&\|\pt \widehat{\pv \W}\|_{L^2_\xi}=t^{-\frac{3}{2}}\|\W''\|_{L^2_v}\leq M t^{-\frac{5}{4}-\delta}.
\end{align*}
Similarly, 
\begin{align*}
    \| \pt v\W\|_{L^2_v}=&t^{-\frac{3}{2}}\|\pv (vW)\|_{L^2_v}\leq M t^{-\frac{3}{2}}
\end{align*}
and 
\begin{align*}
    \| \pt v\W'\|_{L^2_v}=&t^{-\frac{3}{2}}\|\pv^2 (vW)\|_{L^2_v}\leq M t^{-\frac{5}{4}-\delta}.
\end{align*}
\end{proof}
We are now ready to compare $\uap$ with $u_{asymp}$. The next lemma insures that we can replace $u_{asymp}$ by $u_{app}$ in the proof of Proposition~\ref{Prop AC}.
\begin{lemma} \label{u_asymp \uap lemma}
If $W$ is as in \eqref{Wcond}, then
    \begin{align*}
    \|u_{asymp}-u_{app}\|_{L^2_x}\lesssim Mt^{-\frac{3}{4}-\frac{\delta}{2}}, \,\,\, \|u_{asymp}-u_{app}\|_{L^\infty}\lesssim M^2t^{-\frac{1}{2}-\frac{\delta}{2}}.
\end{align*}
\end{lemma}

\begin{proof}

\begin{align*}
    \|u_{asymp}-u_{app}\|_{L^2_x}=&\left\|t^{-\frac{1}{2}}\left(W(x/t)e^{-\frac{ix}{2t}|W\left(\frac{x}{t}\right)|^2\log t}-\W(x/t)e^{-\frac{ix}{2t}|\W\left(\frac{x}{t}\right)|^2\log t}\right)\right\|_{L^2_x}\\
    =&\left\|W(v)e^{-\frac{ix}{2t}|W\left(v\right)|^2\log t}-\W(v)e^{-\frac{ix}{2t}|\W\left(v\right)|^2\log t}\right\|_{L^2_v}\\
    \leq&\left\|W(v)[e^{-\frac{ix}{2t}|W\left(v\right)|^2\log t}-e^{-\frac{ix}{2t}|\W\left(v\right)|^2\log t}]\right\|_{L^2_v}+\|W(v)-\W(v)\|_{L^2_v}.
\end{align*}
The first term can be bounded by 
\begin{align*}
    \|W\|_{L^\infty}\left\|\frac{x}{2t}|W(v)|^2\log t-\frac{x}{2t}|\W(v)|^2\log t\right\|_{L^2_v},
\end{align*}
since $e^{iA}$ for $A\in\R$ is Lipschitz continuous. Then,
\begin{align*}
    \||W(v)|^2\frac{x}{2t}\log t-|\W(v)|^2\frac{x}{2t}\log t\|_{L^2_v} \leq&|\log t|[\|vW\|_{L^\infty}\|\bar{W}-\bar{\W}\|_{L^2_v}+\|v\bar{\W}\|_{L^\infty}\|W-\W\|_{L^2_v}]\\
    \lesssim&M^2t^{-\frac{3}{4}-\delta}|\log t|.
\end{align*}
The second term is bounded similarly by Lemma \ref{Wbounds_lemma}.
We can do the exact same thing to get the $L^\infty$ bound. 
\begin{align*}
    \|u_{asymp}-u_{app}\|_{L^\infty}
    \leq&t^{-\frac{1}{2}}\left\|W(v)[e^{-\frac{ix}{2t}|W\left(v\right)|^2\log t}-e^{-\frac{ix}{2t}|\W\left(v\right)|^2\log t}]\right\|_{L^\infty}+\|W(v)-\W(v)\|_{L^\infty}.
\end{align*}
The first term can be bounded by 
\begin{align*}
t^{-\frac{1}{2}}\|W\|_{L^\infty}\left\|\frac{x}{2t}|W(v)|^2\log t-\frac{x}{2t}|\W(v)|^2\log t\right\|_{L^\infty},
\end{align*}
where
\begin{align*}
\||W(v)|^2\frac{x}{2t}\log t-|\W(v)|^2\frac{x}{2t}\log t\|_{L^\infty}  \leq&|\log t|[\|vW\|_{L^\infty}\|\bar{W}-\bar{\W}\|_{L^\infty}+\|v\bar{\W}\|_{L^\infty}\|W-\W\|_{L^\infty}]\\
    \lesssim&M^2t^{-\frac{1}{2}-\delta}|\log t|.
\end{align*}
\end{proof}

\subsection*{Equation for the difference and well-posedness of the backwards problem}
In order to get from $W$ to our initial data $u_0,$ we need to find an equation for the difference $U:=u-\uap$, where $u$ is the solution to \eqref{dnls}. Then $U$ solves 
\begin{align} \label{U eqn 1}
(i\pt+\px^2)U=N(U,u_{app})-f,
\end{align}
where 
\[N(U,u_{app})=-i\px(U|U|^2+U^2\bu_{app}+2|U|^2u_{app}+2U|u_{app}|^2+\bU u_{app}^2),\]
and $f$ is the error
\begin{align} \label{f defn}
    f:=(i\pt+\px^2)u_{app}+i\px(u_{app}|u_{app}|^2).
    \end{align}
Proving Proposition \ref{Prop AC} is equivalent to solving the problem \eqref{U eqn 1} backwards from infinity, 
  \begin{align}
\label{Ueqn}
    (i\pt+\px^2)U=N(U,u_{app})-f, \,\,\, U(\infty)=0,
\end{align}
as in the following lemma:
\begin{lemma}\label{backwards WP lemma} Let $W$ be as in \eqref{Wcond}. Then the equation  \eqref{Ueqn} is globally well-posed in $H^{1,1}$.
     \end{lemma}
The rest of the section is devoted to the proof of this lemma.

\subsection{Outline of the proof of Asymptotic Completeness}

The main idea of the proof of the well-posedness is obtaining energy estimates via a bootstrap argument and the Klainerman vector field method. The outline is as follows:

\textbf{I. The bootstrap argument}
We make the following bootstrap assumptions

\begin{equation}\label{bootstrap assumptions}
    \begin{split}
         \|U\|_{L^2_x}\leq DMt^{-\frac{1}{2}+\delta}, \, \|LU\|_{L^2}\leq DMt^{-\frac{1}{4}+\delta},\\
      \|U_x\|_{L^2_x}\leq DMt^{-\frac{1}{4}+\delta},   \|\px LU\|_{L^2_x}\leq DMt^{-\frac{1}{4}+\delta}.
    \end{split}
\end{equation}
We will use Klainerman-Sobolev estimates to get 
\begin{align}\label{U L^infty}
    \|U\|_{L^\infty}\leq DMt^{-7/8+\delta}
\end{align}
and
\begin{align}\label{U_x L^infty}
    \|U_x\|_{L^\infty}\leq DMt^{-\frac{3}{4}+\delta}.
\end{align}
The details for this will be in section \ref{klainerman sob sec}.

\textbf{II. Preliminary lemmas}
Within the energy estimates, we will need bounds on various quantities coming from $W$ in the equation \eqref{Ueqn}.
We begin with bounds for $u_{app}$, which appears in the coefficients in \eqref{Ueqn}.

\begin{lemma} \label{\uap bounds lemma}
Let $W$ be as in \eqref{Wcond}. Then we have the following bounds on $\uap$:
\[\|\uap\|_{L^\infty}\lesssim Mt^{-\frac{1}{2}}, \,\,\, \|(\uap)_x\|_{L^\infty}\lesssim Mt^{-\frac{1}{2}},   \,\,\,   \|(\uap)_x\|_{L^2_x}\lesssim M, \,\, \|(\uap)_{xx}\|_{L^2_x}\lesssim Mt^{-\frac{3}{4}-\delta}.\]
Also,
\begin{align*}
    \|L\uap\|_{L^\infty}\lesssim& Mt^{-\frac{1}{2}}(1+M^2|\log t|),\,\,\,
      \|L\uap\|_{L^2_x}\lesssim M(1+M^2|\log t|),\\
      \|\px L\uap\|_{L^2_x}\lesssim& M^3\log t, \,\,\, \|\px L\uap\|_{L^\infty}\lesssim Mt^{-\frac{1}{2}},\,\,\,
        \|\px^2 L\uap\|_{L^2_x}\lesssim M\log t.
\end{align*}
The next lemma is critical, as it provides good $L^2$ type bounds for 
the source term in the equation \eqref{Ueqn}. Here it is essential that we obtain integrable decay at infinity:
\end{lemma}
\begin{lemma}\label{f bound lemma}  Let $W$ be as in \eqref{Wcond}, and $f$ be as in \eqref{f defn}. Then we have
\begin{align*}
    \|f\|_{L^2_x}\lesssim Mt^{-\frac{3}{2}}\log t,  \,\,\,  \|\px f\|_{L^2_x}\lesssim Mt^{-\frac{5}{4}}\log t.
\end{align*}
Also,
\begin{align*}
\|Lf\|_{L^2_x}\lesssim Mt^{-\frac{5}{4}} ,\,\,\, \|\px Lf\|_{L^2_x}\lesssim Mt^{-\frac{5}{4}}\log t.
\end{align*}
\end{lemma}
The proof of both of these lemmas are long computations, and are included in \autoref{proof of uap lemma appendix} and \autoref{f bound proof} respectively.

\textbf{III. Energy estimates for $U$}

\begin{lemma}\label{U EE lemma} Let $U$ be a solution to \eqref{Ueqn}, satisfying the bootstrap assumptions \eqref{bootstrap assumptions}. Then $U$ satisfies the following energy bounds 
\begin{align}\label{U L^2 EE}
     \frac{d}{dt}\|U\|_{L^2}^2
    \lesssim& \lp DMt^{-\frac{5}{4}}+M^2t^{-1}\rp\|U\|_{L^2}^2+Mt^{-\frac{3}{2}+\delta}\|U\|_{L^2},\\
        \frac{d}{dt}\|U_x\|_{L^2}^2
    \lesssim& \lp D^2M^2t^{-1} \rp \|U_x\|_{L^2}^2+ \lp (D^2M^3+M)t^{-\frac{5}{4}+\delta}\rp \|U_x\|_{L^2}.
\end{align}
\end{lemma}
This is proven in section \ref{EE U sec}.

\textbf{IV. Energy estimates for $LU$}
\begin{lemma}\label{LU EE lemma}
  Let $U$ be a solution to \eqref{Ueqn} and $L$ be as in \eqref{def:L}. Then $LU$ satisfies the following energy estimates
\begin{align*}
    \frac{d}{dt}\|LU\|_{L^2}^2\leq&  D^2M^2t^{-1} \|LU\|_{L^2}^2  + \Big( (DM^5+M)t^{-\frac{5}{4}+\delta}
\Big) \|LU\|_{L^2},\\
    \frac{d}{dt}\|\px LU\|_{L^2}^2\lesssim& M^2t^{-1} \|\px LU\|_{L^2}^2+ (DM^3+M)t^{-\frac{5}{4}+\delta}\|\px LU\|_{L^2}.
\end{align*}
\end{lemma}

The proof is in section \ref{L EE}. 

\textbf{V. Closing the bootstrap}
Using a generalized Gr\"onwall argument, the energy estimates lead to 
\begin{align*}
    \|U\|_{L^2}\lesssim Mt^{-\frac{1}{2}+\delta}, \,\, \|U_x\|_{L^2}\lesssim Mt^{-\frac{1}{4}+\delta}, \,\, \|LU\|_{L^2}\lesssim Mt^{-\frac{1}{4}+\delta},\,\, \|\px LU\|_{L^2}\lesssim Mt^{-\frac{1}{4}+\delta}.
\end{align*}
which closes the bootstraps assumptions in Step I.   This is done in section \ref{bootstrap sec}.

\textbf{VI. Constructing solutions}
Now that we have the energy estimates, we can get the existence of solutions to \eqref{Ueqn} using a constructive argument. 
See section \ref{Constructing sols sec}.

\subsection{Klainerman-Sobolev estimates \label{klainerman sob sec}}
The reader may wonder why it is necessary to complete energy estimates for $U$ and $LU$ at both the $L^2$ and the $H^1$ level. It turns out, the estimate for $LU$ in $L^2$ depends on the derivative of $U$ in $L^\infty$.  We will now see how we can bound this quantity, using the Klainerman-Sobolev vector field estimates.

Just as in the linear case, we can write the $L^\infty$ norm of $U$ in terms of energies of $U$ and $LU.$  Using the  Gagliardo-Nirenberg inequality
\begin{align*}
    \|u\|_{L^\infty}\leq C \|Du\|_{L^1},
\end{align*}
and rewriting the equation for $LU$ yields
\begin{align*}
    2ti \px U =&LU-xU,\\
    \frac{d}{dx}|U|^2=&\frac{1}{2ti}[\bU LU -U \widebar{LU}],\\
    \|U\|_{L^\infty}^2\lesssim& t^{-1}\int |U|||LU|\leq \frac{1}{t}\|LU\|_{L^2}\|U\|_{L^2}.
\end{align*}
Applying the bootstrap assumptions \eqref{bootstrap assumptions} gives
\begin{align*}
    \|U\|_{L^\infty}\leq DMt^{-7/8+\delta}.
\end{align*}
Similarly, by applying $L$ to $U_x,$ we get
\begin{align*}
    \frac{d}{dx}|U_x|^2=\frac{1}{2ti}[\bU_x L U_x -U_x \bLU_x].
\end{align*}
Then,
\begin{align*}
    \|U_x\|_{L^\infty}\lesssim& t^{-\frac{1}{2}} \|U_x\|_{L^2}^{\frac{1}{2}} \|LU_x\|_{L^2}^{\frac{1}{2}}\\
    \lesssim& DM t^{-\frac{3}{4}+\delta}.
\end{align*}

\subsection{Energy Estimates}
The plan is now to prove the energy estimates for $U$ and $LU$ given in Lemmas \ref{U EE lemma} and \ref{LU EE lemma}.  These will be done both at the $L^2$ level and the $H^1$ level. 

\subsubsection{\texorpdfstring{Energy Estimate for $U$}{Energy Estimate for U} \label{EE U sec}}

Here we will prove Lemma \ref{U L^2 EE}, the energy estimates for $U$.  From equation \eqref{U eqn 1}, by doing the energy estimate in the usual way, we get
\begin{align}
    \frac{d}{dt}\|U\|_{L^2}^2=&\underbrace{-2\Rea \int i(N(U, u_{app}))\widebar{U}}_I\underbrace{+2\Rea \int i f \widebar{U}}_{II}.
\end{align}
We now decompose the expression into two parts. For the first, we estimate each component separately in terms of $U$ and $\uap$.
\begin{align*}
   I=&-2\Rea \int \px (U|U|^2 +U^2\bu_{app}+2|U|^2u_{app}+2U|u_{app}|^2+\bU u_{app}^2)\bU\\
    =&\frac{1}{2} \int \px(|U|^4)-2\Rea \int  U^2\bU_x\buap+2|U|^2U_x\buap +2U\bU_x|\uap|^2+\bU \bU_x\uap^2\\
    =&-2\Rea \int \px(|U|^2U)\buap -\int \px(|U|^2)|\uap|^2-\Rea \int \px (\bU^2)\uap^2,\\
   |I| \leq&\left|2\Rea \int|U|^2U \px\buap \right|+\int |U|^2\px|\uap|^2+\left|\Rea \int (\bU^2) \px(\uap^2)\right|\\
    \leq& \|U\|_{L^\infty}\|\px(\uap)\|_{L^\infty}\|U\|_{L^2}^2+\|\uap\|_{L^\infty}\|\px(\uap)\|_{L^2}\|U\|_{L^2}^2.
\end{align*}
 For the second part, we apply Lemma \ref{f bound lemma} to control the source term $f$,
\begin{align*}
    |II|\leq \|f\|_{L^2_x}\|U\|_{L^2_x}\lesssim& Mt^{-\frac{3}{2}+\delta}\|U\|_{L^2_x}.
\end{align*}
Overall this gives us
\begin{align} \label{New U EE}
    \frac{d}{dt}\|U\|_{L^2}^2\lesssim&  \lp \|U\|_{L^\infty}\|\px(\uap)\|_{L^\infty}+ \|\uap\|_{L^\infty}\|\px(\uap)\|_{L^\infty} \rp \|U\|_{L^2}^2+\|f\|_{L^2}\|U\|_{L^2} \\
    \lesssim& \lp DMt^{-\frac{5}{4}}+M^2t^{-1}\rp\|U\|_{L^2}^2+Mt^{-\frac{3}{2}+\delta}\|U\|_{L^2}.
\end{align}

\subsubsection{\texorpdfstring{$H^1$ Energy estimate for $U$}{H1 Energy estimate for U} \label{H^1 EE U sec}}
We now turn to the $H^1$-level energy estimate. Differentiating \eqref{U eqn 1} and carrying out the standard energy estimate calculation gives
\begin{align*}
    \frac{d}{dt}\|U_x\|_{L^2}^2=\underbrace{-2\Rea \int i \px (N(U, \uap)) \cd \bU_x}_I +\underbrace{2\Rea \int i \px f \cd \bU_x}_{II}.
\end{align*}
We again split the quantity into two pieces.  For the first, we must be careful not to place more than one derivative on any copy of $U.$  

\begin{align*}
      I=&-2\Rea \int \px^2 (U|U|^2 +U^2\bu_{app}+2|U|^2u_{app}+2U|u_{app}|^2+\bU u_{app}^2)\bU_x\\
       =&\underbrace{2\Rea \int \px(U|U|^2)\bU_{xx}}_{I_a}
       +\underbrace{2\Rea \int \px (U^2\buap+2|U|^2\uap )\bU_{xx}}_{I_b}+\underbrace{2\Rea \int \px (2U|\uap|^2+\bU \uap^2)\bU_{xx}}_{I_c}.
\end{align*}
Splitting the quantity into three components, we check each one separately. The first one can be estimated directly as
\begin{align*}
   I_a=&-2\int \px|U|^2|U_x|^2+\px (U^2)  (\bU_x^2),\\
        |I_a|\lesssim& \|U_x\|_{L^\infty}\|U\|_{L^\infty}\|U_x\|_{L^2}^2.
\end{align*}
For the second one, we first expand out the derivative:
\begin{align*}
    I_b =&2\Rea \int (2UU_x \buap+ 2U_x\bU \uap+ 2U\bU_x\uap+U^2 \px (\buap)+ 2|U|^2 \px (\uap))\bU_{xx}.
\end{align*}
Observe that the complex conjugate does not affect the bound, since we ultimately take absolute values. Thus, the first three terms share the same upper bound, as do the last two. It therefore suffices to check
\begin{align*}
    2\Rea \int UU_x \buap \bU_{xx}=\int U \buap \px (|U_x|^2)=-\int (U_x \buap +U \px(\buap) ) |U_x|^2
\end{align*}
and 
\begin{align*}
    2\Rea \int U^2 \px( \buap) \bU_{xx}=-2\Rea \int (2UU_x \px( \buap) + U^2  \px^2( \buap) )\bU_{x}.
\end{align*}
We have
\begin{align*}
    |I_b|\lesssim& \lp \|U_x\|_{L^\infty} \|\uap\|_{L^\infty}+ \|U\|_{L^\infty}\|\px( \buap) \|_{L^\infty}\rp \|U_x\|_{L^2}^2\\
    +& \lp \|U\|_{L^\infty} \|U_x\|_{L^\infty} \|\px(\uap)\|_{L^2} + \|U\|_{L^\infty}^2\|\px^2(\uap)\|_{L^2} \rp\|U_x\|_{L^2}.
\end{align*}
Similarly for the third component, we only need to bound
\begin{align*}
    I_c\approx& 2\Rea \int( U_x |\uap|^2+ U \uap \px (\uap)) \bU_{xx}\\
    =& \int |\uap|^2 \px( |U_x|^2) - 2\Rea \int (U_x \uap \px(\uap)+ U \px(\uap) \px(\uap)+U \uap \px^2(\uap))\bU_x,\\
    |I_c|\lesssim& \|\uap\|_{L^\infty} \|\px (\uap)\|_{L^\infty}\|U_x\|_{L^2}^2+  (\|\px (\uap)\|_{L^\infty}^2\|U\|_{L^2}+\|U\|_{L^\infty}\|\uap\|_{L^\infty}\|\px^2 (\uap)\|_{L^2})\|U_x\|_{L^2}.
\end{align*}
For the second piece, we again use Lemma \ref{f bound lemma}.
\[|II|\leq \|\px f\|_{L^2_x}\|U_x\|_{L^2}.\]
Combining the above bounds and applying the bootstrap assumptions \eqref{bootstrap assumptions}, we obtain the energy estimate
\begin{align*}
    \frac{d}{dt}\|U_x\|_{L^2}^2\lesssim& \Big[\|U_x\|_{L^\infty}\|U\|_{L^\infty} +  \|U_x\|_{L^\infty} \|\uap\|_{L^\infty}\\
    &+ \|U\|_{L^\infty}\|\px( \buap) \|_{L^\infty}+  \|\uap\|_{L^\infty} \|\px (\uap)\|_{L^\infty} \Big]\|U_x\|_{L^2}^2 \\
    +& \Big[ \|U\|_{L^\infty} \|U_x\|_{L^\infty} \|\px(\uap)\|_{L^2} + \|U\|_{L^\infty}^2\|\px^2(\uap)\|_{L^2} + \|\px (\uap)\|_{L^\infty}^2\|U\|_{L^2}\\
    &+\|U\|_{L^\infty}\|\uap\|_{L^\infty}\|\px^2 (\uap)\|_{L^2} + \|\px f\|_{L^2_x}\Big]\|U_x\|_{L^2}\\
    \lesssim& \lp D^2M^2t^{-\frac{3}{2}}+DM^2t^{-\frac{5}{4}}+M^2t^{-1} \rp \|U_x\|_{L^2}^2\\
    &+ \lp D^2M^3t^{-\frac{3}{2}}+M^2t^{-\frac{5}{4}+\delta}+ Mt^{-\frac{5}{4}+\delta}\rp \|U_x\|_{L^2}\\
    \lesssim& M^2t^{-1}  \|U_x\|_{L^2}^2+ Mt^{-\frac{5}{4}+\delta} \|U_x\|_{L^2}.
\end{align*}
\subsubsection{\texorpdfstring{$L^2$ Energy estimates for $LU$}{L2 Energy estimates for LU} \label{L EE}}
We next derive the equation satisfied by $LU$. Applying $L$ to \eqref{U eqn 1} and using the properties of $L$ from Lemma \ref{properties of L}, we find
\begin{align}
(i\pt+\px^2)LU=L(N(U,u_{app}))-Lf.
\label{LU eqn}
\end{align}
Notice that $N=-i \px(N_1)$ where $N_1=U|U|^2+U^2\bu_{app}+2|U|^2u_{app}+2U|u_{app}|^2+\bU u_{app}^2$. Using the commutator of $L$ with $\px,$ we have
\begin{align*}
    L(N(U,u_{app}))
    =&-i\px L(N_1) +iN_1.
\end{align*}
Expanding $L(N_1)$ term-by-term, we write $L(N_1)=A+B+C+D+E,$ where
\begin{align*}
  A:=  L(U^2\bU)=&2|U|^2LU-U^2\widebar{LU},\\
  B:=  L(U \buap U)=&2U \buap LU- U^2 \widebar{L\uap},\\
  C:=  L(U\bU \uap)=&\bU \uap LU-U \uap \bLU + |U|^2 L\uap,\\
 D:=   L(U \uap \buap)=&|\uap|^2LU - U \uap \widebar{L\uap} + U \buap L \uap,\\
  E:=  L(\bU \uap^2)=& 2 \uap \bU L \uap - \uap^2 \bLU.
\end{align*}
Now we get the energy estimate
\begin{align*}
  \frac{d}{dt} \|LU\|_{L^2}=\underbrace{-2\Rea \int\px \lp L(N_1)\rp \bLU}_I \underbrace{-2\Rea \int  \px \lp N_1 \rp \bLU}_{II} + \underbrace{2\Rea i \int L f \bLU}_{III}.
\end{align*}
We now split the right-hand side into three parts. For the first part, we apply integration by parts to move derivatives and reveal the favorable structure. 
\begin{align*}
    I= 2\Rea \int \lp L(N_1)\rp \px\bLU=2\Rea \int \lp A+B +C +D +E\rp \px\bLU.
\end{align*}
We also split it into five components and estimate each piece.
\begin{align*}
   2\Rea \int A \px\bLU=& 2\Rea \int  2|U|^2LU\px\bLU-U^2\widebar{LU} \px\bLU =\int 2|U|^2 \px |LU|^2- 2\Rea U^2 \px (\bLU)^2,\\
      \left| 2\Rea \int A \px\bLU\right|\lesssim& \|U\|_{L^\infty}\|U_x\|_{L^\infty}\|LU\|_{L^2}^2.
\end{align*}
Again ignoring irrelevant complex conjugates, we see
\begin{align*}
   2\Rea \int (B+C) \px\bLU\approx& 2\int U \buap \px (|LU|^2)-\Rea U^2 \widebar{L\uap} \px\bLU \\
   =&-2\int \px (U \buap )|LU|^2-\Rea \px (U^2 \widebar{L\uap} )\bLU, \\
   \leq&\left[\|U_x\|_{L^\infty}\|\uap\|_{L^\infty}+\|U\|_{L^\infty}\|(\uap)_x\|_{L^\infty}\right]\|LU\|_{L^2}^2\\
   +&\left[\|U\|_{L^\infty}\|U_x\|_{L^\infty}\|L\uap\|_{L^2}+ \|U\|_{L^\infty}^2\|\px( L\uap)\|_{L^2}\right]\|LU\|_{L^2}
\end{align*}
and
\begin{align*}
    2\Rea \int( D+E) \px\bLU \approx& 2\Rea \int  |\uap|^2LU  \px\bLU- U \uap \widebar{L\uap} \px\bLU \\
        =& \int  |\uap|^2 \px|LU|^2- 2\Rea \int U \uap \widebar{L\uap} \px\bLU \\
        =&-\int  \px|\uap|^2 |LU|^2+ 2\Rea \int (U \uap \widebar{L\uap})_x \bLU, \\
        \left|   2\Rea \int (D+E) \px\bLU\right|\lesssim& \|\uap\|_{L^\infty}\|\px \uap\|_{L^\infty}\|LU\|_{L^2}^2
        + \Big[\|U_x\|_{L^\infty}\|\uap\|_{L^\infty}\|L\uap\|_{L^2}\\
        +&\|U\|_{L^\infty}\|\px(\uap)\|_{L^\infty}\|L\uap\|_{L^2}+\|U\|_{L^\infty}\|\px L\uap\|_{L^\infty}\|\uap\|_{L^2}\Big]\|LU\|_{L^2}.
\end{align*}
The second part is expanded to separate each factor, allowing us to estimate each piece individually. 
\begin{align*}
|II|\lesssim    \Big[ \|U\|_{L^\infty}\|U_x\|_{L^\infty}\|U\|_{L^2}+\|U\|_{L^\infty}\|U_x\|_{L^\infty} \|\uap\|_{L^2}+\|U\|_{L^\infty}^2\|\px \uap\|_{L^2} \\
    +\|U_x\|_{L^\infty}\|\uap\|_{L^\infty}\|\uap\|_{L^2}   + \|U\|_{L^\infty}\|\uap\|_{L^\infty}\|\px\uap\|_{L^2}\Big]\|LU\|_{L^2}.
\end{align*}
Lastly, the third term is treated directly using Lemma \ref{f bound lemma}.
\[|III|\lesssim\|Lf\|_{L^2}\|LU\|_{L^2}\lesssim Mt^{-\frac{5}{4}}\|LU\|_{L^2}.\]
Altogether, we have
\begin{align*}
    \frac{d}{dt}\|LU\|_{L^2}^2 \lesssim& \Big(\|U\|_{L^\infty}\|U_x\|_{L^\infty}+\|U_x\|_{L^\infty}\|\uap\|_{L^\infty}\\
    +&\|U\|_{L^\infty}\|(\uap)_x\|_{L^\infty}+\|\uap\|_{L^\infty}\|\px \uap\|_{L^\infty}
\Big)\|LU\|_{L^2}^2\\
+& \Big(\|U\|_{L^\infty}\|U_x\|_{L^\infty}\|L\uap\|_{L^2}+ \|U\|_{L^\infty}^2\|\px( L\uap)\|_{L^2}+\|U_x\|_{L^\infty}\|\uap\|_{L^\infty}\|L\uap\|_{L^2}\\
+& \|U\|_{L^\infty}\|\px(\uap)\|_{L^\infty}\|L\uap\|_{L^2}+\|U\|_{L^\infty}\|\px L\uap\|_{L^\infty}\|\uap\|_{L^2}\\
+&\|U\|_{L^\infty}\|U_x\|_{L^\infty}\|U\|_{L^2}+\|U\|_{L^\infty}\|U_x\|_{L^\infty} \|\uap\|_{L^2}+\|U\|_{L^\infty}^2\|\px \uap\|_{L^2} \\
    +&\|U_x\|_{L^\infty}\|\uap\|_{L^\infty}\|\uap\|_{L^2}   + \|U\|_{L^\infty}\|\uap\|_{L^\infty}\|\px\uap\|_{L^2}+\|Lf\|_{L^2}\Big)\|LU\|_{L^2}.
\end{align*}
Applying the bootstraps \eqref{bootstrap assumptions}, the pointwise bounds \eqref{U L^infty}, \eqref{U_x L^infty}, and Lemma \ref{\uap bounds lemma}, we get
\begin{align*}
        \frac{d}{dt}\|LU\|_{L^2}^2\lesssim& \lp D^2M^2t^{-\frac{3}{2}}+DM^2 t^{-\frac{5}{4}}+M^2t^{-1}\rp \|LU\|_{L^2}^2\\
        +& \Big( D^2M^2t^{-\frac{3}{2}}\cd M(1+M^2|\log t|)+D^2M^2t^{-\frac{3}{2}}\cd Mt^{-\frac{1}{2}-\frac{\delta}{2}}+2DM^3t^{-\frac{5}{4}}\\
        +&2DM^5t^{-\frac{5}{4}}|\log t|+DM^3 t^{-\frac{3}{4}-1-\frac{\delta}{2}}+DMt^{-\frac{3}{4}}Mt^{-\frac{1}{2}}\cd M+Mt^{-\frac{5}{4}}
\Big) \|LU\|_{L^2}\\
\lesssim& \lp D^2M^2t^{-1}\rp \|LU\|_{L^2}^2  + \Big( (DM^5+M)t^{-\frac{5}{4}+\delta}
\Big) \|LU\|_{L^2}.
\end{align*}

\subsubsection{\texorpdfstring{$H^1$ energy estimate for $LU$}{H1 energy estimate for LU}}
Lastly, we need to prove the $H^1$-level energy estimate for $LU.$ The following lemma will be key, as it provides a way to control terms with second derivatives, by using $L$ to transfer a derivative onto a different factor.
\begin{lemma}\label{u_xx to L lemma}
    For the vector field $L$ defined in \eqref{def:L}, and any two functions $u,v$ we have 
    \begin{align}\label{u_xx to L eqn}
        u_{xx}\bv = \frac{1}{2ti}Lu_x\bv-\frac{1}{2ti}\widebar{Lv}u_x-\bv_x u_x.
    \end{align}
\end{lemma}
\begin{proof}
From applying $L$ to $v$ and $u_x,$ we have
\begin{align*}
Lv=&xv+2tiv_x, \,\,\,
    Lu_x=xu_x+2tiu_{xx},\\
     \frac{x}{2ti}\bv=&\frac{1}{2ti}\widebar{Lv}+\bv_x, \,\,\,
    u_{xx}=\frac{1}{2ti}Lu_x-\frac{x}{2ti}u_x.
\end{align*}
Then, plugging each of these in yields
\begin{align*}
    u_{xx}\bv=&\frac{1}{2ti}Lu_x\bv-\frac{x}{2ti}\bv u_x\\
        =&\frac{1}{2ti}Lu_x\bv-\lp\frac{1}{2ti}\widebar{Lv}+\bv_x\rp u_x\\
       =& \frac{1}{2ti}Lu_x\bv-\frac{1}{2ti}\widebar{Lv}u_x-\bv_x u_x.
\end{align*}
\end{proof}
Note that we also have $ u_{x}\bv = \frac{1}{2ti}Lu\bv-\frac{1}{2ti}\widebar{Lv}u-\bv_x u,$ though in  this paper we only make use of the second derivative version.

Now we compute the $H^1$-level energy estimate for $LU$, starting by differentiating \eqref{LU eqn}:
\begin{align}
(i\pt+\px^2)\px LU=-i\px^2 L(N_1) +i\px N_1-\px Lf.
\end{align}
This lead to the energy identity
\begin{align*}
  \frac{d}{dt} \|\px LU\|_{L^2}
  =&\underbrace{-2\Rea \int \lp \px^2 L(N_1)\rp \px\bLU}_I \underbrace{-2\Rea \int  \lp \px N_1 \rp \px \bLU}_{II} -\underbrace{2\Rea i \int \px L f \px\bLU}_{III}.
\end{align*}
Once again, we split the quantity into three parts, using the same notation as in the $L^2$ energy estimate of $L(N_1)=A+B+C+D+E$. We have
\begin{align*}
   \int \px A \px^2 \bLU  
   \approx& \int \px(|U|^2)LU\px^2 \bLU+|U|^2\px LU\px^2 \bLU-\px (U^2)\widebar{LU}\px^2 \bLU- U^2\px\widebar{LU}\px^2 \bLU\\
        =&\int\underbrace{ -\px^2(|U|^2)LU\px \bLU+\px^2 (U^2)\widebar{LU}\px \bLU}_{A_1}
     -\underbrace{2\px|U|^2 (|\px LU|^2)+2\px U^2(\px\widebar{LU})^2}_{A_2}.
\end{align*}
The second piece can be bounded directly
\begin{align*}
    \left| \int A_2\right|\lesssim \|U\|_{L^\infty}\|U_x\|_{L^\infty}\|\px LU\|_{L^2}^2.
\end{align*}
For $A_1,$ we use Lemma \ref{u_xx to L lemma} to transfer one derivative, followed by the commutator $[\px, L]=1.$
\begin{align*}
    2 \Rea \int A_1\approx& -2 \Rea \int 2\Rea (U_{xx}\bU+ |U_x|^2)LU\px \bLU\\
    =&-2\Rea \int 2 \Rea \lp \frac{1}{2ti}LU_x\bU-\frac{1}{2ti}\widebar{LU}U_x-\bU_x U_x + |U_x|^2\rp LU\px \bLU\\
      \lesssim&\frac{1}{t}\left| \int   \lp (\px LU-U)\bU-\widebar{LU}U_x\rp LU\px \bLU\right|.
      \end{align*}
      It then becomes necessary to control $LU$ in $L^\infty$.  This can be handled using the Gagliardo-Nirenberg $\|LU\|_{L^\infty}\leq \|LU\|_{H^1}$, leading to
      \begin{align*}
    \lesssim& \frac{1}{t}\lp \|U\|_{L^\infty}\|LU\|_{L^\infty}\|\px LU\|_{L^2}^2+ \lp \|U\|_{L^\infty}^2\|LU\|_{L^2}+\|U_x\|_{L^\infty}\|LU\|_{L^\infty}\|LU\|_{L^2} \rp \|\px LU\|_{L^2}\rp\\
    \lesssim& D^2M^2t^{-2} \|\px LU\|_{L^2}^2  +D^3M^3t^{-\frac{9}{4}+2\delta} \|\px LU\|_{L^2}.
\end{align*}
For $B$ and $C,$ disregarding terms that differ by a complex conjugate, we have
\begin{align*}
     &2\Rea \int \px (B+C) \px^2 \bLU \\
     \approx &\Rea \int \px (U \buap LU+|U|^2 L\uap ) \px^2 \bLU  \\
        = &\Rea \int (U_x \buap LU+U (\buap)_x LU+U \buap\px LU+\px|U|^2 L\uap+|U|^2 \px L\uap ) \px^2 \bLU  \\
= &-\Rea \int (U_{xx} \buap LU+2U_x \px\buap LU+2U_x \buap \px LU+U (\buap)_{xx} LU+U \buap\px^2 LU\\
&+2U (\buap)_x \px LU
+2\Rea(U_{xx}\bU +|U_x|^2) L\uap+\px|U|^2 \px L\uap+|U|^2 \px^2 L\uap ) \px \bLU  \\
=&B_1+B_2,
\end{align*}
where 
\begin{align*}
    B_1=&-\Rea \int ( 2U_x \px\buap LU+2U_x \buap \px LU +2U (\buap)_x \px LU+U \buap\px^2 LU\\
    &\hspace{1.1cm}+\px|U|^2 \px L\uap+|U|^2 \px^2 L\uap )\px \bLU,\\
    B_2=&-\Rea \int (U_{xx} \buap LU +U (\buap)_{xx} LU+2\Rea(U_{xx}\bU +|U_x|^2) L\uap ) \px \bLU.
\end{align*}
$B_1$ can be estimated directly, whereas $B_2$ relies on Lemma \ref{u_xx to L lemma}.
\begin{align*}
 |B_1|\lesssim&\Big( \|U_x\|_{L^\infty}\|\px\uap\|_{L^\infty} \|LU\|_{L^2} + \|\px L\uap\|_{L^\infty}\|U_x\|_{L^\infty} \| U\|_{L^2}+ \|U\|_{L^\infty}^2 \| \px^2L\uap\|_{L^2}
                  \Big)\|\px LU\|_{L^2}\\
   +&  \Big(  \|U_x\|_{L^\infty}\|\uap\|_{L^\infty} +  \|U\|_{L^\infty} \|\px\uap\|_{L^\infty}   \Big)\|\px LU\|_{L^2}^2\\
   \lesssim& DM^2t^{-\frac{5}{4}+\delta}\|\px LU\|_{L^2}^2+ D^2M^3t^{-\frac{3}{2}+2\delta }\|\px LU\|_{L^2}.
\end{align*}
Applying Lemma \ref{u_xx to L lemma} to each second derivative,
\begin{align*}
    B_2
    =&-\Rea \int \Bigg( \lp \frac{1}{2ti}LU_x\buap-\frac{1}{2ti}\widebar{L\uap}U_x-\px\buap U_x \rp LU\\
    +&2\Rea\lp\frac{1}{2ti}LU_x\bU-\frac{1}{2ti}\widebar{LU}U_x\rp L\uap\Bigg) \px \bLU\\
    &+\lp \frac{1}{2ti}L (\uap)_x\bU-\frac{1}{2ti}\widebar{LU}\px \uap-\bU_x \px\uap\rp \widebar{LU} \px LU\\
    =&-\Rea \int \Bigg( \lp \frac{1}{2ti}(\px LU-U)\buap-\frac{1}{2ti}\widebar{L\uap}U_x-\px\buap U_x \rp LU\\
    +&2\Rea\lp\frac{1}{2ti}(\px LU-U)\bU-\frac{1}{2ti}\widebar{LU}U_x\rp L\uap\Bigg) \px \bLU\\
    &+\lp \frac{1}{2ti}(\px L\uap-\uap)\bU-\frac{1}{2ti}\widebar{LU}\px \uap-\bU_x \px\uap\rp \widebar{LU} \px LU.\\
 |B_2|\lesssim&t^{-1} \Big( \|U\|_{L^\infty}\|\uap\|_{L^\infty}\|LU\|_{L^2}+\|U_x\|_{L^\infty}\|L\uap\|_{L^\infty}\|LU\|_{L^2}+ \|U\|_{L^\infty}^2 \|L\uap\|_{L^2}\\
 +& \|U\|_{L^\infty}\|\px L\uap \|_{L^\infty}\|LU\|_{L^2}+\|\px \uap\|_{L^\infty}\|LU\|_{L^\infty}  \|LU\|_{L^2}   \Big)   \|\px LU\|_{L^2}\\
&+\lp\|U_x\|_{L^\infty}\|\px\uap\|_{L^\infty}\|LU\|_{L^2} \rp   \|\px LU\|_{L^2}\\
&+t^{-1}\lp\|LU\|_{L^\infty}\|\uap\|_{L^\infty} +  \|U\|_{L^\infty} \|L\uap\|_{L^\infty}\rp \|\px LU\|_{L^2}^2\\
\lesssim& DM^2t^{-\frac{7}{4}+\delta}\|\px LU\|_{L^2}^2 + D^2M^3t^{-\frac{3}{2}+2\delta}\|\px LU\|_{L^2}.
\end{align*}
For the last components of the first term $I$, we use a similar strategy, breaking it up and separating the pieces that need Lemma \ref{u_xx to L lemma}.
\begin{align*}
     -2\Rea \int (D+E) \px^3 \bLU  \approx&-2\Rea \int \px^2(   |\uap|^2LU + U \buap L \uap  ) \px \bLU=: D_1+D_2.
     \end{align*}
     \begin{align*}
     D_1\approx& \Rea \int \lp  \px^2 \uap \buap LU+ |\px \uap|^2 LU \rp \px \bLU\\
     &+ |\uap|^2 \px^2 LU \px \bLU+ \px \uap \buap \px LU  \px \bLU\\
     =&\Rea \int  \lp \frac{1}{2ti}L(\uap)_x\buap-\frac{1}{2ti}\widebar{L\uap}\px \uap \rp LU  \px \bLU\\
     &+(- \px |\uap|^2 + \px \uap \buap) |\px LU|^2.\\
     D_2\approx& \int \Big( U_{xx}\buap L\uap+ U_x \px \buap L\uap +U_x \buap \px L\uap\\
     +& U \px \buap \px L\uap +  U \px^2 \buap L\uap+ U \buap \px^2 L \uap   \Big) \px \bLU  \\
     =&   \int \Big( \lp \frac{1}{2ti}LU_x\buap-\frac{1}{2ti}\widebar{L\uap}U_x-\px\buap U_x \rp  L\uap+ U_x \px \buap L\uap \\
     +&U_x \buap \px L\uap  + U \px \buap \px L\uap +  U \px^2 \buap L\uap+ U \buap \px^2 L \uap   \Big) \px \bLU. \\
     |D_1|\lesssim& t^{-1}\lp\|\px L\uap\|_{L^\infty}\|\uap\|_{L^\infty}+ \|L\uap\|_{L^\infty}\|\px\uap\|_{L^\infty} 
 \rp\| LU\|_{L^2} \|\px LU\|_{L^2}\\
 +& \|\px\uap\|_{L^\infty}\|\uap\|_{L^\infty} \|\px LU\|_{L^2}^2.\\
 |D_2|\lesssim& t^{-1} \lp \|\uap\|_{L^\infty}\|L\uap\|_{L^\infty} \|\px LU\|_{L^2}^2+\|L\uap\|_{L^\infty}^2\|U_x\|_{L^2}\|\px LU\|_{L^2}   \rp  \\
 +& \Big(  \|U_x\|_{L^\infty} \|\uap\|_{L^\infty}\|\px L\uap\|_{L^2}+  \|U\|_{L^\infty} \|\px\uap\|_{L^\infty}\|\px L\uap\|_{L^2}\\
 +&\|U\|_{L^\infty} \|\px^2\uap\|_{L^2}\| L\uap\|_{L^\infty}+ \|U\|_{L^\infty} \|\uap\|_{L^\infty}\|\px^2 L\uap\|_{L^2}   \Big) \|\px LU\|_{L^2}.\\
 |D_1+D_2|\lesssim& M^2t^{-1} \|\px LU\|_{L^2}^2+ DM^3 t^{-\frac{5}{4}}   \|\px LU\|_{L^2}.
\end{align*}
So far, we have a bound for the first part,
\begin{align*}
|I|\lesssim& M^2t^{-1} \|\px LU\|_{L^2}^2+ D^2M^3 t^{-\frac{5}{4}}   \|\px LU\|_{L^2}.
\end{align*}
For the second part, we expand out the derivative and estimate each piece.
\begin{align*}
|II|\lesssim&    \Big[ \|U\|_{L^\infty}\|U_x\|_{L^\infty}\|U\|_{L^2}+\|U\|_{L^\infty}\|U_x\|_{L^\infty} \|\uap\|_{L^2}+\|U\|_{L^\infty}^2\|\px \uap\|_{L^2} \\
    & \hspace{0.2cm}+\|U_x\|_{L^\infty}\|\uap\|_{L^\infty}\|\uap\|_{L^2}   + \|U\|_{L^\infty}\|\uap\|_{L^\infty}\|\px\uap\|_{L^2}\Big]\|\px LU\|_{L^2}\\
    \lesssim&D^3M^3t^{-2+2\delta}+ DM^3t^{-\frac{5}{4}+\delta}\|\px LU\|_{L^2}.
\end{align*}
We get the bound for the third part directly from Lemma \ref{f bound lemma},
\begin{align*}
|III|\lesssim& M t^{-\frac{5}{4}+\delta}   \|\px LU\|_{L^2}.
\end{align*}
Therefore we obtain
\begin{align*}
    \frac{d}{dt} \|\px LU\|_{L^2}^2\lesssim& M^2t^{-1} \|\px LU\|_{L^2}^2+ (DM^3+M)t^{-\frac{5}{4}+\delta}\|\px LU\|_{L^2}.
\end{align*}

\subsection{Closing the bootstrap argument\label{bootstrap sec}}

In order to close the bootstraps, we need to use the Nonlinear version of Gr\"onwall's inequality, Lemma \ref{NL Gronwall}.
This is applied to each energy estimate, using the final condition $U(\infty)=0,$ where $f(t)$ below refers to the function $f$ in Lemma \ref{NL Gronwall}.
\begin{itemize}
    \item $f(t)=\|U\|_{L^2}$:
    Recall from Lemma \ref{U EE lemma} we had
\begin{align*}
      \frac{d}{dt}\|U\|_{L^2}^2
    \lesssim& \lp DMt^{-\frac{5}{4}}+M^2t^{-1}\rp\|U\|_{L^2}^2+Mt^{-\frac{3}{2}+\delta}\|U\|_{L^2}\\
    \lesssim&\lp DM^2t^{-1}\rp\|U\|_{L^2}^2+Mt^{-\frac{3}{2}+\delta}\|U\|_{L^2}.
\end{align*}
Applying Lemma \ref{NL Gronwall} gives
\begin{align*}
  \|U\|_{L^2}\leq& \frac{M}{1-2\delta-DM^2} t^{-\frac{1}{2}+\delta}.
\end{align*}
This closes the bootstrap argument, since  $\frac{M}{1-2\delta-DM^2} <DM$ for $M, \delta$ small enough. 

\item $f(t)=\|U_x\|_{L^2}$: The $H^1$ energy estimate obtained was
\begin{align*}
    \frac{d}{dt}\|U_x\|_{L^2}^2
    \lesssim& \lp D^2M^2t^{-1} \rp \|U_x\|_{L^2}^2+ \lp (D^2M^3+M)t^{-\frac{5}{4}+\delta}   \rp \|U_x\|_{L^2}.
\end{align*}
By applying the lemma, we get
\begin{align*}
    \|U\|_{L^2}\leq& \frac{D^2M^3+M}{\frac{1}{2}-2\delta-D^2M^2} t^{-\frac{1}{4}+\delta}.
\end{align*}

\item $f(t)=\|LU\|_{L^2}$:
Similarly, the energy estimate
\begin{align*}
    \frac{d}{dt}\|LU\|_{L^2}^2\lesssim&\lp D^2M^2t^{-1}\rp \|LU\|_{L^2}^2  + \Big( (DM^5+M)t^{-\frac{5}{4}+\delta}
\Big) \|LU\|_{L^2},
\end{align*}
becomes
\begin{align*}
    \|LU\|_{L^2}\leq \frac{DM^5+M}{\frac{1}{2}-2\delta-D^2M^2}t^{-\frac{1}{4}+\delta}.
\end{align*}

\item $f(t)=\|\px LU\|_{L^2}$:
Lastly,
\begin{align*}
    \frac{d}{dt} \|\px LU\|_{L^2}^2\lesssim& M^2t^{-1} \|\px LU\|_{L^2}^2+ (DM^3+M)t^{-\frac{5}{4}+\delta}\|\px LU\|_{L^2}
\end{align*}
leads to
\begin{align*}
    \|\px LU\|_{L^2}\leq \frac{DM^3+M}{\frac{1}{2}-2\delta-M^2}t^{-\frac{1}{4}+\delta}.
\end{align*}

\end{itemize}

For $M^2\ll \delta\ll 1$, each coefficient is strictly smaller than $DM,$ so each bootstrap closes, thus completing the energy bounds.  

\subsection{Constructing solutions \label{Constructing sols sec}}
At this stage, we have the energy estimates for the function $U$, assuming that $U$ exists, but these do no immediately guarantee the existence of the solution $U.$ To construct a solution, we proceed via a limiting argument.

Let $u_T$ be the solution to \eqref{dnls} with final time data $u(T)=\uap(T)$, solved backwards in time from $t=T$ to $t=0.$  The existence of $u_T$ is guaranteed by the standard well-posedness theory for \eqref{dnls}.  Now define $U_T=u_T-\uap.$  Then  $U_T$ solves

\begin{align*}
    (i\pt+\px^2)U_T=N(U_T,u_{app})-f, \,\, U_T(T,x)=0.
\end{align*}

We want to show that $U_T\to U$ in $L^2.$

\subsubsection{\texorpdfstring{Energy estimate for $U_T$}{Energy estimate for UT}}

We can see that $U_T$ will satisfy the same energy estimate as $U,$ which for $w(t)=\|U_T\|_{L^2}$ leads to 
\begin{align*}
    w(t)\leq& w(t_0) \cdot \lp \frac{t}{t_0}\rp^{M^2}+ \left|\frac{M}{-1+2\delta+2M^2}\lp t^{-\frac{1}{2}+\delta}-t_0^{-\frac{1}{2}+\delta+M^2}t^{-M^2} \rp\right|.
\end{align*}
  Then, since $U_T(T)=0,$ setting $t_0=T$ yields the bound
\begin{align*}
    \|U_T\|_{L^2}\lesssim&M\left| t^{-\frac{1}{2}+\delta}-T^{-\frac{1}{2}+\delta} \lp \frac{T}{t}\rp^{\frac{M^2}{2}}\right|.
\end{align*}
We proceed by showing that $\{U_T\}$ is a Cauchy sequence. 
Let $U_{T_1}, U_{T_2}$ be solutions to
\begin{align} \label{U_T_1 eqn}
(i\pt+\px^2)U_{T_1}=N(U_{T_1},u_{app})-f, \,\,\, U_{T_1}(T_1,x)=0,\\
    (i\pt+\px^2)U_{T_2}=N(U_{T_2},u_{app})-f, \,\,\, U_{T_2}(T_2,x)=0, \label{U_T_2 eqn}
\end{align}
respectively. 
We need to show that 
$\|U_{T_1}-U_{T_2}\|_{L^2}\to 0$ as $T_1, T_2\to \infty$. Without loss of generality, let $T_1>T_2$.
There are two cases to consider, depending the relative size of $t$ and $T_2.$

\textbf{Case 1: $T_2<t$.}
By the triangle inequality,
\begin{align*}
    \|U_{T_1}-U_{T_2}\|_{L^2}\leq \|U_{T_1}\|_{L^2}+ \|U_{T_2}\|_{L^2}.
\end{align*}
From the energy estimate, and the fact that $T_2<t$, we have the bound 
\begin{align*}
    \|U_{T_1}\|_{L^2}\leq& Mt^{-\frac{1}{2}+\delta}+MT_1^{-\frac{1}{2}+\delta}\lp \frac{T_1}{t}\rp^{M^2}\\
    \leq& MT_2^{-\frac{1}{2}+\delta}+MT_1^{-\frac{1}{2}+\delta+M^2}T_2^{-M^2}.
\end{align*}
Similarly for $U_{T_2},$
\begin{align*}
    \|U_{T_2}\|_{L^2}\leq& Mt^{-\frac{1}{2}+\delta}-MT_2^{-\frac{1}{2}+\delta}\lp \frac{T_2}{t}\rp^{M^2}\\
    \leq& MT_2^{-\frac{1}{2}+\delta}-MT_2^{-\frac{1}{2}+\delta}.
\end{align*}
Therefore $\|U_{T_1}\|_{L^2},\|U_{T_2}\|_{L^2}\to 0$ as $T_1,T_2\to \infty.$

\textbf{Case 2: $t<T_2<T_1$.} First, let us denote the difference by $V:=U_{T_2}-U_{T_1}$.  Our goal is to determine the equation satisfied by $V.$ Subtracting \eqref{U_T_1 eqn} from \eqref{U_T_2 eqn}, we obtain
\begin{align*}
(i\pt+\px^2)V=N(U_{T_2},u_{app})-N(U_{T_1},u_{app}).
\end{align*}
This difference can be calculated as
\begin{align*}
    N(U_{T_2},u_{app})-N(U_{T_1},u_{app})
    =&i\px(N_1+N_2+N_3+N_4+N_5),
\end{align*}
where
\begin{align*}
    N_1=&U_{T_1}|U_{T_1}|^2-U_{T_2}|U_{T_2}|^2=-(|V|^2V+2|U_{T_1}|^2V+U_{T_1}^2\Bar{V}+V^2 \bU_{T_1}+2U_{T_1}|V|^2),\\
    N_2=&U_{T_1}^2\bu_{app}-U_{T_2}^2\bu_{app}=-\buap (2U_{T_1}V+V^2),\\
    N_3=&2|U_{T_1}|^2u_{app}-2|U_{T_2}|^2u_{app}=-2\uap (V\bU_{T_1}+ \bar{V}U_{T_1}+|V|^2),\\
    N_4=&2U_{T_1}|u_{app}|^2-2U_{T_2}|u_{app}|^2=-2|u_{app}|^2V,\\
    N_5=&\bU_{T_1} u_{app}^2-\bU_{T_2} u_{app}^2=-u_{app}^2\Bar{V}.
\end{align*}
The energy estimate for $V$ is thus
\begin{align*}
    \frac{d}{dt}\|V\|_{L^2}^2=&2\Rea \int \px (N_1+...+N_5) \bar{V}\\
    =&-2\Rea \int  (N_1+...+N_5) \px\bar{V}.
\end{align*}
Note that $N_1$ is exactly the nonlinearity $N,$ except with $V$ in place of $U$ and $U_{T_1}$ in place of $\uap.$ Consequently, the same structure as in the energy estimate for $U,$ holds here.  In particular, the $N_1$ term can be bounded by 
\begin{align*}
  \left|\Rea \int  N_1\px\bar{V}\right|\lesssim&  \lp\|V\|_{L^\infty}\|\px(U_{T_1})\|_{L^\infty}+ \|U_{T_1}\|_{L^\infty}\|\px(U_{T_1})\|_{L^\infty}\rp \|V\|_{L^2}^2.
\end{align*}
$N_2$ and $N_3$ have the same base structure when ignoring complex conjugates. It follows that
\begin{align*}
\Rea \int ( N_2+N_3)\px\bar{V}\approx&   2\Rea \int \buap (2U_{T_1}V+V^2) \bar{V}_x=2 \int \buap U_{T_1}\px |V|^2 + 2\Rea \int \buap V^2 \bar{V}_x\\
\lesssim&\lp \|\px \uap\|_{L^\infty}\|U_{T_1}\|_{L^\infty}+ \|\uap\|_{L^\infty}\|\px(U_{T_1})\|_{L^\infty}\rp \|V\|_{L^2}^2\\
&+\|\px\uap\|_{L^\infty}\|V\|_{L^\infty}\|V\|_{L^2}^2.
\end{align*}
The final two pieces can be estimated directly. 
\begin{align*}
     \left|\Rea \int  (N_4+N_5)\px\bar{V}\right|\lesssim \|\px\uap\|_{L^\infty}\|\uap\|_{L^\infty}\|V\|_{L^2}^2.
\end{align*}
Overall, this gives
\begin{align*}
     \frac{d}{dt}\|V\|_{L^2}^2\lesssim&  \Big( \|V\|_{L^\infty}\|\px(U_{T_1})\|_{L^\infty}+\|\px(U_{T_1})\|_{L^\infty}\|U_{T_1}\|_{L^\infty} +   \|\px \uap\|_{L^\infty}\|U_{T_1}\|_{L^\infty} \\
     &+ \|\uap\|_{L^\infty}\|\px(U_{T_1})\|_{L^\infty}+\|\px\uap\|_{L^\infty}\|V\|_{L^\infty}+ \|\px\uap\|_{L^\infty}\|\uap\|_{L^\infty}     \Big) \|V\|_{L^2}^2\\
     \lesssim& M^2t^{-1} \|V\|_{L^2}^2.
\end{align*}
An application of Gr\"onwall’s inequality gives, for $f(t)=\|V\|_{L^2}^2,$
\begin{align*}
    f(t)\leq& f(t_0) \exp \int_{t_0}^t M^2 s^{-1}\, ds\\
    f(t)\leq&f(t_0) \cd \lp \frac{t}{t_0}\rp^{M^2},
\end{align*}
hence
\begin{align*}
     \|U_{T_2}-U_{T_1}\|_{L^2}\leq  \|(U_{T_2}-U_{T_1})(t_0)\|_{L^2}\cd \lp \frac{t}{t_0}\rp^{\frac{M^2}{2}}.
\end{align*}
Using $t_0=T_2,$ the fact that $t<T_2,$ and the energy estimate for $U_{T_1},$ we have
\begin{align*}
     \|U_{T_2}-U_{T_1}\|_{L^2}\leq&  \|(U_{T_2}-U_{T_1})(T_2)\|_{L^2}\cd \lp \frac{t}{T_2}\rp^{\frac{M^2}{2}}\\
     \leq&  \|U_{T_1}(T_2)\|_{L^2}\\
    \leq&  M\left| T_2^{-\frac{1}{2}+\delta}-T_1^{-\frac{1}{2}+\delta+M^2}T_2^{-M^2}\right|\to 0,
\end{align*}
as $T_1,T_2\to \infty.$  This gives the existence and therefore completes the proof of the asymptotic completeness.

\appendix

\section{Proof of Lemma \ref{\uap bounds lemma} \label{proof of uap lemma appendix}}
Recall the definition of the ansatz
\begin{align*}
        u_{app}:=t^{-\frac{1}{2}}e^{i\frac{x^2}{4t}}\W\left(\frac{x}{t}\right)e^{-\frac{ix}{2t}|\W\left(\frac{x}{t}\right)|^2\log t}.
\end{align*}
Throughout this appendix, we adopt some shorthand to simplify the notation. We use $'$ to denote derivatives with respect to the variable $v,$ and set $h:=i\frac{x^2}{4t}-\frac{ix}{2t}|\W\left(\frac{x}{t}\right)|^2\log t$, so that $h$ is the phase appearing in $\uap$. It is worth noting that derivatives with respect to $x$ and $v$ are related by $\px=\frac{1}{t}\pv,$ and their $L^2$ norms satisfy $\|\cdot \|_{L^2_x}=t^{\frac{1}{2}}\|\cdot \|_{L^2_v}$.

\begin{proof}[Proof of Lemma \ref{\uap bounds lemma}]
The first bound follows from the definition of $\uap.$  For the second one, we calculate the derivative to see
\begin{align*}    
(\uap)_x =&t^{-\frac{1}{2}}e^{h}\left[\frac{1}{t}\W'(v)+\frac{i}{2}v\W(v)
    -\frac{i}{2t}|\W|^2\W(v)\log t-
  it^{-1}v\W(v)\cdot \Rea(\W' \bW)\log t\right].
\end{align*}
Then using Lemma \ref{Wbounds_lemma}, we have
\begin{align*}
    \|(\uap)_x\|_{L^\infty}\lesssim Mt^{-\frac{1}{2}}
\end{align*}
and
\begin{align*}
     \|(\uap)_x\|_{L^2_x}\lesssim t^{-\frac{1}{2}}\|vW\|_{L^2_x}= \|vW\|_{L^2_v}\leq M.
\end{align*}
For the second derivative,
\begin{align*} (\uap)_{xx}
  =&t^{-\frac{3}{2}}e^{h}\Big[\frac{1}{t}\W''(v)+\frac{i}{2}\W(v)+\frac{i}{2}v\W'(v)
    -\frac{i}{2t}\log t(2\W \W'\bW+\W^2\bW')\\
    -&
it^{-1}\log t(\W(v)\cdot \Rea(\W' \bW)+v\W'(v)\cdot \Rea(\W' \bW) +v\W(v)\cdot \Rea(\W'' \bW+\W' \bW')    \Big],\\
\|(\uap)_{xx}\|_{L^2_x}\lesssim& t^{-1}\Big[t^{-1} \|\W''\|_{L^2_v} +\|\W\|_{L^2_v} +\|v\W'\|_{L^2_v}+t^{-1}\log t \|\W\|_{L^\infty}^2 \|\W'\|_{L^2_v}\\
+&t^{-1}\log t \|\W\|_{L^\infty}^2\|\W'\|_{L^2_v}+\|v\W'\|_{L^\infty}\|\W'\|_{L^\infty}\|\W\|_{L^2_v}+\|v\W\|_{L^\infty}\|\W\|_{L^\infty}\|\W''\|_{L^2_v} \Big]\\
\lesssim& Mt^{-\frac{3}{4}-\delta}.
\end{align*}
For the $L\uap$ bounds, first calculate
\begin{align*}
    L\uap=& t^{-\frac{1}{2}}e^h \left[ 2i\W'(v)
    +|\W|^2\W(v)\log t+
 2v\W(v)\cdot \Rea(\W' \bW)\log t\right].
\end{align*}
Then we take the norms and apply Lemma \ref{Wbounds_lemma}.
\begin{align*}
\|L\uap\|_{L^\infty}\leq&t^{-\frac{1}{2}}\left[\|\W'(v) \|_{L^\infty} +|\log t| \|\W\|_{L^\infty}^3 +|\log t| \| v\W\|_{L^\infty}\| W\|_{L^\infty} \| W'\|_{L^\infty}\right] \\
    \lesssim& Mt^{-\frac{1}{2}}(1+M^2|\log t|), \\
\|L\uap\|_{L^2_x}\leq& \left[ \| W'\|_{L^2_v}+\log t \| W\|_{L^\infty}^2\| W\|_{L^2}+\log t \| vW\|_{L^\infty}\| W\|_{L^\infty}\| W'\|_{L^2}\right]\\
\lesssim& M(1+M^2\log t).
\end{align*}
Now, we need both the $L^\infty$ and $L^2$ norms for $\px L\uap.$ Start by differentiating the previous quantity for $L\uap.$
\begin{align*}
\px L\uap =& \underbrace{(\px h) L\uap}_{I}+ \underbrace{t^{-\frac{1}{2}}e^h\px\left[ 2i\W'(v)
    +|\W|^2\W(v)\log t+
 2v\W(v)\cdot \Rea(\W' \bW)\log t\right]}_{II}.
\end{align*}
We break this up into two parts. For the first, we use
\begin{align*}
     \px h=& \frac{ix}{2t}-\frac{i}{2t}\log t (|\W|^2+x\px|\W|^2),
\end{align*}
in order to write the first term as
\begin{align*}
 I =&t^{-\frac{1}{2}}e^h \left[ -v\W'(v)
    +\frac{iv}{2}|\W|^2\W(v)\log t+
 iv^2\W(v)\cdot \Rea(\W' \bW)\log t\right]\\
 -& \frac{i}{2}t^{-\frac{1}{2}}e^h\log t \left[ 2i\frac{1}{t}|\W|^2\W'(v)
    +\frac{1}{t}\log t |\W|^2|\W|^2\W(v)+
 2\frac{1}{t}\log t |\W|^2v\W(v)\cdot \Rea(\W' \bW)\right]\\
 -&it^{-\frac{1}{2}}e^h \log t\Big[ 2i\W'(v)\frac{v}{t}\Rea(\W'\bW)\\
    &+\frac{v}{t}\Rea(\W'\bW)|\W|^2\W(v)\log t+
 2v\W(v)\cdot \Rea(\W' \bW)\frac{v}{t}\Rea(\W'\bW)\log t\Big].
\end{align*}
By expanding the derivative, second term is
\begin{align*}
  II =&t^{-\frac{3}{2}}e^h\p_v\left[ 2i\W'(v)
    +\W^2\bW\log t+
 2v\W(v)\cdot \Rea(\W' \bW)\log t\right]\\
 =&t^{-\frac{3}{2}}e^h\Big[ 2i\W''(v)
    +(2\W\W'\bW+\W^2\bW')\log t\\
    +&
 (2\W+2v\W')\cdot \Rea(\W' \bW)\log t+ 2\log t v\W \Rea (\W''\bW+|\W'|^2)\Big].
\end{align*}
Next, take the $L^\infty$ norm of each term and apply Lemma \ref{Wbounds_lemma}.
\begin{align*}
  \|I\|_{L^\infty} \lesssim &t^{-\frac{1}{2}}\Big[\|v\W'\|_{L^\infty}+\|v\W\|_{L^\infty}\|\W\|_{L^\infty}^2+t^{-1}\log t(\|\W\|_{L^\infty}^5+\|v\W\|_{L^\infty}^2\|\W'\|_{L^\infty})\\
  +&\frac{1}{t} \|\W'\|_{L^\infty}^2 \|v\W\|_{L^\infty}+ \frac{\log t}{t}\|v\W\|_{L^\infty} \|\W'\|_{L^\infty} \|\W\|_{L^\infty}^3 + \frac{\log t}{t}\|v\W\|_{L^\infty}^2 \|\W'\|_{L^\infty}^2 \|\W\|_{L^\infty}\Big]\\
    \lesssim&t^{-\frac{1}{2}}\Big[M+t^{-1}M^3+\frac{\log t}{t}M^5 \Big]
    \lesssim Mt^{-\frac{1}{2}},\\
     \|II\|_{L^\infty}\lesssim &t^{-\frac{3}{2}}\Big[\|\W''\|_{L^\infty}+\log t \big(\|\W\|_{L^\infty}^2 \|\W'\|_{L^\infty}  + \|\W'\|_{L^\infty}^2\|v\W\|_{L^\infty} \\
     &+   \|v\W\|_{L^\infty} (\|\W''\|_{L^\infty} \|\W\|_{L^\infty} +  \|\W'\|_{L^\infty}^2 )\big)  \Big]\\
\lesssim& Mt^{-1-\frac{\delta}{2}}|\log t|.
\end{align*}
Now, we take the $L^2$ norm and apply Lemma \ref{Wbounds_lemma}. Then,
\begin{equation}\label{px L I}
    \begin{split}
         \|I \|_{L^2_x}\lesssim&\Big[\|v\W'\|_{L^2}+\log t(\|\W\|_{L^\infty}^2\|v\W\|_{L^2}+ \|v\W\|_{L^\infty}^2\|\W'\|_{L^2})+\frac{1}{t} \|\W'\|_{L^\infty} \|v\W\|_{L^\infty}\|\W'\|_{L^2}\\
 +& \frac{\log t}{t}\|v\W\|_{L^\infty}  \|\W\|_{L^\infty}^3 \|\W'\|_{L^2}+ \frac{\log t}{t}\|v\W\|_{L^\infty}^2 \|\W'\|_{L^\infty} \|\W\|_{L^\infty}\|\W'\|_{L^2}\Big]
    \end{split}
\end{equation}
and
\begin{equation}\label{px L II}
    \begin{split}
         \|II \|_{L^2_x}\lesssim&t^{-1}\Big[\|\W''\|_{L^2}+\log t \big(\|\W\|_{L^\infty}^2 \|\W'\|_{L^2}  + \|\W'\|_{L^\infty}\|v\W\|_{L^\infty} \|\W'\|_{L^2} \\
    +& \|v\W\|_{L^\infty} \|\W\|_{L^\infty}\|\W''\|_{L^2} + \|v\W\|_{L^\infty} \|\W'\|_{L^\infty}\|\W'\|_{L^2}\big)  \Big].
    \end{split}
\end{equation}
Combining these yields
\begin{align*}
    \|\px L\uap \|_{L^2_x}\lesssim M^3 \log t +Mt^{-\frac{3}{4}-\delta}+M^3t^{-\frac{3}{4}-\delta} \log t \lesssim M^3 \log t.
\end{align*}
For the second derivative, we proceed by relying on previously established identities as much as possible.  Recall that here we only need the $L^2$ norm. As a first step, we decompose it into four separate terms.
\begin{align*}
\px^2 L\uap =& \underbrace{(\px^2 h) L\uap}_{III}+\underbrace{(\px h) \px L\uap}_{IV}\\
+&\underbrace{(\px h)t^{-\frac{1}{2}}e^h\px\left[ 2i\W'(v)
    +|\W|^2\W(v)\log t+
 2v\W(v)\cdot \Rea(\W' \bW)\log t\right]}_{V}\\
 +& \underbrace{t^{-\frac{1}{2}}e^h\px^2\left[ 2i\W'(v)
    +|\W|^2\W(v)\log t+
 2v\W(v)\cdot \Rea(\W' \bW)\log t\right]}_{VI}.
\end{align*}
For term $III,$ we start by finding the $L^\infty$ norm of $\px^2 h,$ then using the previous $L^2$ bound of $L\uap.$ We have
\begin{align*}
 \px^2 h  
    =&\frac{1}{t}\lp \frac{i}{2}-\frac{i}{2t}\log t \lp 2 \Rea (\W' \bW)+\Rea(\W'\bW)+v\Rea(\W''\bW+|\W'|^2)\rp\rp, \\
    \|\px^2 h\|_{L^\infty}\lesssim& \frac{1}{t}\lp 1+ t^{-1}\log t \lp \|\W'\|_{L^\infty} \|\W\|_{L^\infty}+\|\W''\|_{L^\infty} \|v\W\|_{L^\infty}+  \|v\W'\|_{L^\infty}^2\rp \rp \\
    \lesssim& t^{-1}.
\end{align*}
In $L^2_x, \,III$  is bounded by 
\begin{align*}
\|III\|_{L^2_x}\leq&   \|\px^2 h\|_{L^\infty} \|L\uap\|_{L^2_x}\lesssim t^{-1} M(1+M^2|\log t|).
\end{align*}
We split up $IV$ into two pieces.
\begin{align*}
  IV=  \px h \px L\uap=& \underbrace{\frac{i}{2}v \px L\uap}_{IV_a} + \underbrace{\frac{i}{2t}\log t\lp |\W|^2+v\Rea(\W'\bW) \rp\px L\uap}_{IV_b}.
\end{align*}
 In the first, $IV_a,$ we modify the previous bounds for $\px L\uap$ by attaching an additional $v$ to each term in \eqref{px L I} and \eqref{px L II}.
\begin{align*}
    \|IV_a\|_{L^2_x}\lesssim&\|v^2\W'\|_{L^2}+\log t(\|\W\|_{L^\infty}\|v\W\|_{L^\infty}\|v\W\|_{L^2}+ \|v\W\|_{L^\infty}^2\|v\W'\|_{L^2})\\
   +& t^{-1} \Big[\|\W'\|_{L^\infty} \|v\W\|_{L^\infty}\|v\W'\|_{L^2}+\|v\W''\|_{L^2}+ \log t\|v\W\|_{L^\infty}^2  \|\W\|_{L^\infty}^2 \|\W'\|_{L^2}\\
    &+\log t\big(\|v\W\|_{L^\infty}^3 \|\W'\|_{L^\infty} \|\W'\|_{L^2}+\|\W\|_{L^\infty}\|v\W\|_{L^\infty} \|\W'\|_{L^2}  + \|\W'\|_{L^\infty}\|v\W\|_{L^\infty} \|v\W'\|_{L^2}  \\
    &+\|v\W\|_{L^\infty}^2\|\W''\|_{L^2}+ \|v\W\|_{L^\infty} \|v\W'\|_{L^\infty}\|\W'\|_{L^2}\big)  \Big]\\
    \lesssim&M+M^3\log t + Mt^{-\frac{1}{4}-\frac{\delta}{2}}.
\end{align*}
Also, the second piece can be bounded by
\begin{align*}
        \|IV_b\|_{L^2_x}\leq& \lp t^{-1+\delta}\|\W\|_{L^\infty}^2+ t^{-1}\|v\W\|_{L^\infty}\|\W'\|_{L^\infty} \rp \|\px L\uap\|_{L^2_x}\\
        \leq& M^2t^{-1+\delta}\cdot M^3 \log t\leq M^5 t^{-1+2\delta}.
\end{align*}
For term $V,$ we see that it can be written as
\begin{align*}
V=& (\px h) \cd II\\
=&\underbrace{\frac{i}{2}v\cdot II}_{V_a}+ \underbrace{\lp -\frac{i}{2t}\log t |\W|^2-\log t\frac{iv}{t}\Rea(\W'\bW)\rp \cd II}_{V_b}.
\end{align*}
Then, for the first part, $V_a$ as before we use the previous bound for $II$, but add the $v$ factor.
\begin{align*}
     \|V_a\|_{L^2_x}\leq& t^{-1}\Big[\|v\W''\|_{L^2}+\log t \big(\|v\W\|_{L^\infty}\|\W\|_{L^\infty} \|\W'\|_{L^2}  + \|\W'\|_{L^\infty}\|v\W\|_{L^\infty} \|v\W'\|_{L^2} \\
    +&   \|v\W\|_{L^\infty}^2\|\W''\|_{L^2} + \|v\W\|_{L^\infty} \|\W'\|_{L^\infty}\|v\W'\|_{L^2}\big)  \Big]\\
    \lesssim& Mt^{-\frac{1}{4}-\frac{\delta}{2}}.
\end{align*}
For the second part, $V_b$, we have
\begin{align*}
     \left\|V_b
\right\|_{L^2_x} \leq& t^{-1+\delta}\lp  \|\W\|_{L^\infty}^2+ \|v\W\|_{L^\infty} \|\W'\|_{L^\infty}   \rp   \| II\|_{L^2_x}\\
\lesssim&M^3t^{-\frac{7}{4}+\delta}.
\end{align*}
Lastly, we estimate $VI$ directly to get
\begin{align*}
 VI=&t^{-\frac{5}{2}}e^h\pv^2\left[ 2i\W'(v)
    +|\W|^2\W(v)\log t+
 2v\W(v)\cdot \Rea(\W' \bW)\log t\right],\\
 \|VI\|_{L^2_x}\lesssim& t^{-2}\Big[ \|\W'''\|_{L^2}  + \log t ( \|\W\|_{L^\infty}^2\|\W''\|_{L^2} + 
  \|\W'\|_{L^\infty}^2\|\W\|_{L^2}+ \|v\W'\|_{L^\infty} \|\W'\|_{L^\infty}\|\W'\|_{L^2} \\
  +&\|v\W\|_{L^\infty} \|\W'\|_{L^\infty}\|\W''\|_{L^2}+\|v\W\|_{L^\infty} \|\W\|_{L^\infty}\|\W'''\|_{L^2}  )
 \Big]\\
 \lesssim& Mt^{-\frac{5}{4}}.
\end{align*}

Overall, we obtain
\begin{align*}
    \|\px^2 L\uap\|_{L^2_x}\lesssim M^3|\log t|.
\end{align*}
\end{proof}

\section{Proof of Lemma \ref{f bound lemma} \label{f bound proof} }

\begin{proof}[Proof of Lemma \ref{f bound lemma}]
The objective of this section is to bound the source term $f$, as defined in \eqref{f defn}, in $L^2.$ Recall the definitions
\begin{align*}
     f=(i\pt+\px^2)u_{app}+i\px(u_{app}|u_{app}|^2)
\end{align*}
and 
\begin{align*}
    u_{app}:=t^{-\frac{1}{2}}e^{i\frac{x^2}{4t}}\W\left(\frac{x}{t}\right)e^{-\frac{ix}{2t}|\W\left(\frac{x}{t}\right)|^2\log t}.
\end{align*}
We need to write out $f$ completely in terms of $\W$ to see what cancels. We have 
\begin{equation*}
    \begin{split}
       i\pt u_{app}    =&\frac{e^h}{t^{\frac{1}{2}}}\bigg[\frac{-i}{2t}\W+\frac{v^2}{4}\W -ivt^{-1}\W'(v)-\frac{v}{2t}\W|\W|^2\log t+i \pt \W\\
    &+\frac{v}{2t}\W|\W|^2 -\frac{v^2}{2t}\log t \W(|\W|^2)' +\frac{v}{2}\log t \W\pt |\W|^2\bigg],\\
\px^2(\uap)=&e^ht^{-\frac{1}{2}}\Big[\frac{i}{2t}\W-\frac{v^2}{4} \W+ivt^{-1}\pv \W +\frac{v}{2t}\log t\W |\W|^2+t^{-2}\W''\\
-&\frac{i}{t^2}\log t  \W' (|\W|^2+v\pv|\W|^2)+\frac{v^2}{2t}\log t\W(|\W|^2)'\\
    +&\frac{(\log t)^2}{4t^2} \W (|\W|^2+v\pv |\W|^2)^2-\frac{i}{2t^2}\log t \W (2 \pv |\W|^2+v \pv^2|\W|^2)\Big], \\
    i\px(u_{app}|u_{app}|^2)=&e^ht^{-\frac{1}{2}}\Big[it^{-2}\pv(|\W|^2\W)-\frac{v}{2t}|\W|^2\W\\
    +&\frac{1}{2t^2}\log t |\W|^4\W +\frac{v}{2t^2}\log t |\W|^2\W \pv |\W|^2 \Big].
\end{split}
\end{equation*}
By combining these, we get
\begin{align*}
    f=&e^ht^{-\frac{1}{2}}[f_1+f_2+f_3+f_4+f_5],
\end{align*}
where
\begin{align*}
    f_1=&i \pt \W+v\log t \W \Rea (\bW \pt \W),\\
    f_2=& \frac{(\log t)^2}{4t^2} \W (|\W|^2+v\pv |\W|^2)^2,\\
      f_3=&-\frac{i}{t^2}\log t \pv (\W |\W|^2)- \frac{iv}{2t^2}\log t \W\pv^2|\W|^2, \\
        f_4=&-\frac{iv}{t^2}\log t \pv \W \pv|\W|^2 +t^{-2}\pv^2\W,  \\
          f_5=& it^{-2}\pv(|\W|^2\W) +\frac{1}{2t^2}\log t |\W|^4\W +\frac{v}{2t^2}\log t |\W|^2\W \pv |\W|^2.
\end{align*}
Observe that because of the relation between the $L^2_x$ norm and the $L^2_v$ norm, we have   $\|f\|_{L^2_x}=\|f_1+f_2+f_3+f_4+f_5\|_{L^2_v}.$  We will consider each term $f_i$ separately, and apply Lemma \ref{Wbounds_lemma} to get the $L^2$ bound on $f.$ First,
\begin{align*}
    \|f_1\|_{L^2_v}\lesssim& t^{-1}[\|W_{t^{\frac{1}{2}}}\|_{L^2_v}+\log t\|v\W\|_{L^\infty}\|\W\|_{L^\infty}\|W_{t^{\frac{1}{2}}}\|_{L^2_v}]\\
    \lesssim& Mt^{-\frac{3}{2}}+M^3t^{-\frac{3}{2}}\log t.
\end{align*}
Then, for the rest of the terms, we first rewrite them before taking the norm. We also use the notation $\approx$ to denote ignoring complex conjugates, which don't impact the norm.  This gives
\begin{align*}
f_2 =&  \frac{(\log t)^2}{4t^2} \W (|\W|^4+4v|\W|^2\Rea(\W'\bW)+2v^2\Rea(\W'\bW)^2),\\
    \|f_2\|_{L^2_v}\lesssim& t^{-2}(\log t)^2\lp\|\W\|_{L^2}\|\W\|_{L^\infty}^4+\|\W\|_{L^\infty}^3\|v\W\|_{L^\infty} \|\W'\|_{L^2}+\|v\W\|_{L^\infty}^2\|\W'\|_{L^\infty}\|\W'\|_{L^2}\rp ,\\
    \lesssim& M^5 t^{-2}(\log t)^2,\\
f_3=&-\frac{i}{t^2}\log t \lp  \pv (\W |\W|^2)+\frac{v}{2}\W\pv^2|\W|^2 \rp\\
    \approx&t^{-2 }\log t \lp \W^2\W' +v\W [ \W'' \bW+|\W'|^2]\rp,\\
    \|f_3\|_{L^2_v}\lesssim& t^{-2}\log t\Big[\|\W\|_{L^\infty}\|\W'\|_{L^\infty}\|\W\|_{L^2}+\|v\W\|_{L^\infty}\|\W\|_{L^\infty}\|\W''\|_{L^2_v}+\|v\W\|_{L^\infty}\|\W'\|_{L^\infty}\|\W'\|_{L^2_v}\Big]\\
    \lesssim& M^3 t^{-2}\log t+M^3t^{-\frac{3}{2}},\\
f_4=&-\frac{iv}{t^2}\log t  \W' \pv|\W|^2 +t^{-2}\W'', \\
   \|f_4\|_{L^2_v}\lesssim& t^{-2}\log t\|\W'\|_{L^\infty}\|\W'\|_{L^\infty}\|v\W\|_{L^2}+t^{-2}\|\W''\|_{L^2}\\
\lesssim& M^3t^{-\frac{3}{2}}\log t+Mt^{-2+\frac{1}{4}-\delta},\\
    \|f_5\|_{L^2_v}\lesssim& t^{-2}\|\W\|_{L^\infty}^2\|\W'\|_{L^2_v}+t^{-2}\log t\|\W\|_{L^\infty}^4\|\W\|_{L^2}+t^{-2}\log t\|\W\|^3_{L^\infty}\|v\W\|_{L^\infty}\|\W'\|_{L^2}\\
    \lesssim& M^3t^{-2}+M^5t^{-2}\log t.
\end{align*}
Then we have
\begin{align*}
   \|f\|_{L^2_x}
   \lesssim& Mt^{-\frac{3}{2}+\delta}.
\end{align*}
Next, we turn to estimating $Lf$, and afterward we will prove the $\px f$ bound, making use of the previous two bounds. First rewrite the vector field using the change of variables $v=x/t$:
$L=x+2ti\px=x+2i \p_v$. Then, applying $L$ to $f$ gives
\begin{align*}
    Lf =&xf+2it(\px h)f+ 2it^{-\frac{1}{2}}e^h \p_v (f_1+...+f_5).
\end{align*}
Notice that the $xf$ term above cancels with the following term coming from the derivative of $h$:
\begin{align*}
 2ti   \px h =&-x+\log t |\W|^2+2v\Rea (\W' \bW).
\end{align*}
After this cancellation, what remains is
\begin{align*}
    Lf=&\underbrace{\log t |\W|^2f}_{B_1}+\underbrace{2v\Rea (\W' \bW)f}_{B_2}+\underbrace{2it^{-\frac{1}{2}}e^h \p_v (f_1+...+f_5)}_{B_3}.
    \end{align*}
Splitting this into three terms, we can start by directly estimating the first two
\begin{align*}
    \|B_1\|_{L^2_x}\lesssim&\log t \|\W\|_{L^\infty}^2\|f\|_{L^2_x}\lesssim  M^3t^{-\frac{3}{2}+\delta}\log t,\\
 \|B_2\|_{L^2_x}\lesssim&\|v\W\|_{L^\infty}\|\W'\|_{L^\infty}\|f\|_{L^2}\lesssim M^3t^{-\frac{3}{2}+\delta}.
 \end{align*}
For the third term, first observe that
\begin{align*}
\|B_3\|_{L^2_x}=2\|\pv(f_1+...+f_5)\|_{L^2_v}.
\end{align*}
In order to bound this term, we must calculate the derivative of each $f_i$ component. 
\begin{align*}
    \p_v f_1 =&i\pt \W'+\log t \W \Rea( \W_t \bW)+v\log t \W' \Rea( \W_t \bW)+v\log t \W  \Rea (\W'_t \bW+ \W' \bW_t),\\
    \|\pv f_1\|_{L^2_v}\lesssim& \|\pt\W'\|_{L^2_v}+\log t \Big(\|\W\|_{L^\infty}\|\pt \W\|_{L^2_v}(\|\W\|_{L^\infty}+ \|v\W'\|_{L^\infty})\\
    &\hspace{2.4cm} +\|v\W\|_{L^\infty}(\|\W\|_{L^\infty}\|\pt\W'\|_{L^2_v}+ \|\W'\|_{L^\infty}\|\pt \W\|_{L^2_v}) \Big)\\
  \lesssim& M t^{-\frac{3}{2}}|\log t|,\\
    \p_v f_2 \approx& \frac{(\log t)^2}{4t^2}\Big[\W'\lp|\W|^4+4v|\W|^2\Rea(\W'\bW)+v^2(\W')^2\W^2\rp+ \W^4\W' +\W^3(\W')^2\\
    +&v\W^3\W''\W+v \W^3(\W')^2+v^2\W^2\W'(\W'' \W)+v^2 \W^2(\W')^3    \Big],\\
    \|\pv f_2\|_{L^2_v}\lesssim & \frac{(\log t)^2}{4t^2}\Big[\|\W \|_{L^\infty}^4 \|\W' \|_{L^2}  +   \|v\W \|_{L^\infty} \|\W \|_{L^\infty}^2 \|\W' \|_{L^\infty} \|\W' \|_{L^2} + \|\W \|_{L^\infty}^2\|\W' \|_{L^\infty}^2 \|\W' \|_{L^2}  \\
    +& \|v\W \|_{L^\infty} \|\W \|_{L^\infty}^3\|\W'' \|_{L^2} +\|v\W \|_{L^\infty}^2\|\W \|_{L^\infty}\|\W '\|_{L^\infty} \|\W'' \|_{L^2}  +\|v\W \|_{L^\infty}^2\|\W '\|_{L^\infty}^2  \|\W' \|_{L^2}      \Big]\\
    \lesssim& M^5(\log t)^2 t^{-\frac{7}{4}-\delta},\\
    \pv f_3\approx& -i t^{-2}\log t \Big [ (\W')^2\W+\W^2 \W''+2\W  \Rea(\W''\bW+|\W'|^2) \\
    +&v\lp 2\W'\Rea (\W''\bW+|\W'|^2)+2\W \Rea(\W'''\bW+2\W''\bW')  \rp    \Big],\\
    \|\pv f_3\|_{L^2_v}\lesssim&t^{-2}\log t \Big [ \|\W \|_{L^\infty} \|\W'\|_{L^\infty}\|\W' \|_{L^2}+\|\W \|_{L^\infty}^2\|\W'' \|_{L^2} +  \|\W \|_{L^\infty}^2 \|\W'' \|_{L^2}\\
    +&\|\W \|_{L^\infty} \|\W'\|_{L^\infty} \|\W' \|_{L^2} +  \|v\W \|_{L^\infty}  \|\W' \|_{L^\infty}\|\W'' \|_{L^2}+ \|v\W' \|_{L^\infty} \|\W' \|_{L^\infty}\|\W' \|_{L^2}\\
    +&  \|v\W \|_{L^\infty} \|\W \|_{L^\infty} \|\W''' \|_{L^2}+ \|\W \|_{L^\infty}\|\W' \|_{L^\infty}  \|\W'' \|_{L^2} \Big]\\
    \lesssim& M^3t^{-\frac{5}{4}-\delta}\log t,\\
    \p_v f_4\approx& -\frac{2i}{t^2}\log t \lp \W' \Rea(\W'\bW) +v(\W'' \Rea(\W'\bW) + \W'' \Rea(\W''\bW) +\W'' |\W|^2    ) \rp   +\frac{1}{t^2}\p_v^3\W, \\
      \|\pv f_4\|_{L^2_v}  \lesssim& t^{-2}\log t \|\W'\|_{L^\infty}\lp   \|\W\|_{L^\infty} \|\W'\|_{L^2} + \|v\W\|_{L^\infty} \|\W''\|_{L^2} + \|v\W'\|_{L^\infty} \|\W'\|_{L^2}\rp+ t^{-2}\|\W'''\|_{L^2}\\
 \lesssim& t^{-2}\log t(M^3+M^3t^{\frac{1}{4}-\delta})+Mt^{-\frac{5}{4}},\\
    \pv f_5\approx& t^{-2}(\W |\W'|^2+ \W^2 \W'')+t^{-2}\log t [ \W^4 \W'+ v\W^3 (\W')^2+ v\W^4 \W''],\\
     \|\pv f_5\|_{L^2_v}\lesssim&t^{-2}(M^3+M^3t^{\frac{1}{4}-\delta})+t^{-2}\log t(M^5+M^5t^{\frac{1}{4}-\delta})\\
     \lesssim& M^3t^{-\frac{5}{4}}.
\end{align*}
Overall, we get
\begin{align*}
    \|Lf\|_{L^2_x}\lesssim Mt^{-\frac{5}{4}}.
\end{align*}

\subsection{\texorpdfstring{$\px f$ bound}{partial f bound}}
Since we already have a bound for $Lf$, we can use it to estimate $\px f.$
We rearrange the equation for $L$ in terms of $\px:$
\begin{align*}
    Lf=(x+2ti \px)f, \,\,\,\px f= \frac{1}{2ti}[Lf-xf].
\end{align*}
Taking the $L^2$ norm, it follows that
\begin{align*}
        \|\px f\|_{L^2_x}\lesssim \frac{1}{t}  \|Lf\|_{L^2_x}+  \|v f\|_{L^2_x}\lesssim Mt^{-\frac{9}{4}}+ \|v f\|_{L^2_x}. 
\end{align*}
For this, we multiply each term in the calculation of $f$ by $v$.  This produces
\begin{align*}
    vf_1\approx& \frac{1}{t}vW_{t^{\frac{1}{2}}}+v^2\frac{1}{t}\log t \W^2 W_{t^{\frac{1}{2}}},\\
    \|vf_1\|_{L^2_v}\lesssim& t^{-1}[\|vW_{t^{\frac{1}{2}}}\|_{L^2_v}+\log t\|v\W\|_{L^\infty}\|v\W\|_{L^\infty}\|W_{t^{\frac{1}{2}}}\|_{L^2_v}]\\
    \lesssim& Mt^{-\frac{3}{2}}+M^3t^{-\frac{3}{2}}\log t,\\
vf_2 =&  \frac{(\log t)^2}{4t^2} v\W (|\W|^4+4v|\W|^2\Rea(\W'\bW)+2v^2\Rea(\W'\bW)^2),\\
    \|vf_2\|_{L^2_v}\lesssim& t^{-2}\log t \|v\W\|_{L^\infty}\lp\|\W\|_{L^2}\|\W\|_{L^\infty}^3+\|\W\|_{L^\infty}^2\|v\W\|_{L^\infty}\|\W'\|_{L^2}+\|v\W'\|_{L^\infty}^2\|\W\|_{L^2} \rp \\
    \lesssim& M^5 t^{-2}\log t,\\
    vf_3
    \approx&t^{-2}\log t v\W^2\W' +\log t v^2\W t^{-2}[ \W'' \bW+|\W'|^2],\\
    \|vf_3\|_{L^2_v}\lesssim& t^{-2}\log t\Big[\|v\W\|_{L^\infty}\|\W'\|_{L^\infty}\|\W\|_{L^2}+\|v\W\|_{L^\infty}^2\|\W''\|_{L^2_v}+\|v\W\|_{L^\infty}\|\W'\|_{L^\infty}\|v\W'\|_{L^2_v}\Big]\\
    \lesssim& M^3 t^{-\frac{5}{4}}\log t,\\
vf_4=& v^2t^{-2}\log t \W' \W' \W+ t^{-2}v\W'',\\
   \|vf_4\|_{L^2_v}\lesssim& t^{-2+\delta}\|v\W'\|_{L^\infty}\|\W'\|_{L^\infty}\|v\W\|_{L^2}+t^{-2}\|v\W''\|_{L^2}\\
\lesssim& M^3t^{-\frac{3}{2}+\delta},\\
  v f_5 =&it^{-2}v\p_v(|\W|^2\W) +\frac{1}{2t^2}\log t v|\W|^4\W +\frac{1}{2t^2}v^2\log t |\W|^2\W \p_v |\W|^2,\\
    \|vf_5\|_{L^2_v}\lesssim& t^{-2}\|v\W\|_{L^\infty}\|\W\|_{L^\infty}\|\W'\|_{L^2_v}+t^{-2+\delta}\lp\|\W\|_{L^\infty}^3\|v\W\|_{L^\infty}\|\W\|_{L^2}+\|\W\|^2_{L^\infty}\|v\W\|_{L^\infty}^2\|\W'\|_{L^2}\rp\\
    \lesssim& M^3t^{-2}+M^5t^{-2}\log t.
\end{align*}
Therefore,
\begin{align*}
        \|\px f\|_{L^2_x}\lesssim Mt^{-\frac{5}{4}}\log t.
\end{align*}

\subsection{\texorpdfstring{$\px Lf$}{partial Lf}}
Lastly, we prove the $L^2$ bound on the derivative of $Lf$. From our calculation before, we have
\begin{align*}
    Lf=&\log t |\W|^2f+2v\Rea (\W' \bW)f+2it^{-\frac{1}{2}}e^h \p_v (f_1+...+f_5)\\
    =&B_1+B_2+B_3.
    \end{align*}
Differentiating leads to
\begin{align*}
      \px  Lf=&\px (B_1+B_2+B_3),
\end{align*}
where we calculate and bound each term separately.
\begin{align*}
    \px B_1=&\log t \px f \cdot |\W|^2+ \log t \cd f \cd \frac{1}{t}\pv |\W|^2,\\
    \| \px B_1\|_{L^2_x}\lesssim& \log t \|\W\|_{L^\infty}^2\|\px f\|_{L^2_x}+t^{-1}\log t \|\W\|_{L^\infty}\|\W'\|_{L^\infty}\|f\|_{L^2_x}\\
    \lesssim& M^3 t^{-\frac{5}{4}}\log t,\\
    \px B_2=& 2v\Rea (\W' \bW)\px f+ 2t^{-1}\Rea (\W' \bW)f+ 2t^{-1}v\Rea (\W'' \bW+|\W'|^2)f,\\
     \| \px B_2\|_{L^2_x}\lesssim& \|v\W\|_{L^\infty}\|\W'\|_{L^\infty} \| \px f\|_{L^2_x}+t^{-1}\lp \|\W\|_{L^\infty}\|\W'\|_{L^\infty}+\|v\W\|_{L^\infty}\|\W''\|_{L^\infty} + \|\W'\|_{L^\infty}^2 \rp\|  f\|_{L^2_x}\\
     \lesssim& M^2 t^{-\frac{5}{4}}\log t.
\end{align*}
For the third term, we split it up further into two pieces.
\begin{align*}
    \px B_3=& \underbrace{2it^{-\frac{1}{2}}(\px h) e^h\pv (f_1+...+f_5)}_{B_4}+ \underbrace{2it^{-\frac{3}{2}} e^h \pv^2(f_1+...+f_5)}_{B_5}.
\end{align*}
We see that $B_4$ can be written in terms of $B_3$ as
\begin{align*}
    B_4=&(\px h) \cd B_3\\
    =&  \underbrace{\frac{i}{2}vB_3}_{B_{4a}}-\underbrace{\frac{i}{2t}\log t (|\W|^2+v\pv|\W|^2)B_3}_{B_{4b}}.
\end{align*}
First, we multiply all the previous bounds for $B_3$ by $v.$ Recall
\begin{align*}
\|B_3\|_{L^2_x}=&2\|\pv(f_1+...+f_5)\|_{L^2_v},
\end{align*}
so we must calculate
\begin{align*}
 \|v\pv f_1\|_{L^2_v}\lesssim& \|v\pt\W'\|_{L^2_v}+\log t \Big(\|v\W\|_{L^\infty}\|\pt \W\|_{L^2_v}(\|\W\|_{L^\infty}+ \|v\W'\|_{L^\infty})\\
    &\hspace{2.4cm} +\|v\W\|_{L^\infty}(\|v\W\|_{L^\infty}\|\pt\W'\|_{L^2_v}+ \|v\W'\|_{L^\infty}\|\pt \W\|_{L^2_v}) \Big)\\
    \lesssim&Mt^{-\frac{5}{4}-\delta}+ \log t (M^4t^{-\frac{3}{2}}+M^3t^{-\frac{5}{4}-\delta} + M^3 t^{-\frac{3}{2}}  ),  \\
 \|v\pv f_2\|_{L^2_v}\lesssim & \frac{(\log t)^2}{4t^2}\Big[\|\W \|_{L^\infty}^3 \|v\W \|_{L^\infty}\|\W' \|_{L^2}  +   \|v\W \|_{L^\infty}^2 \|\W \|_{L^\infty} \|\W' \|_{L^\infty} \|\W' \|_{L^2} \\
 +& \|\W \|_{L^\infty}\|v\W \|_{L^\infty}\|\W' \|_{L^\infty}^2 \|\W' \|_{L^2}   + \|v\W \|_{L^\infty}^2 \|\W \|_{L^\infty}^2\|\W'' \|_{L^2}\\
 +&\|v\W \|_{L^\infty}^3\|\W '\|_{L^\infty} \|\W'' \|_{L^2}  +\|v\W \|_{L^\infty}^2\|\W '\|_{L^\infty}^2  \|v\W' \|_{L^2}      \Big]\\
    \lesssim& M^5t^{-2+\frac{1}{4}}\log t, \\
     \|v\pv f_3\|_{L^2_v}\lesssim&t^{-2}\log t \Big [ \|v\W \|_{L^\infty} \|\W'\|_{L^\infty}\|\W' \|_{L^2}+\|\W \|_{L^\infty}\|v\W \|_{L^\infty}\|\W'' \|_{L^2}  \\
    +& \|\W \|_{L^\infty}\|v\W \|_{L^\infty} \|\W'' \|_{L^2} + \|v\W \|_{L^\infty} \|\W'\|_{L^\infty} \|\W' \|_{L^2}+ \|v\W \|_{L^\infty}  \|v\W' \|_{L^\infty}\|\W'' \|_{L^2}\\
    +&\|v\W' \|_{L^\infty}^2\|\W' \|_{L^2}+  \|v\W \|_{L^\infty}^2 \|\W''' \|_{L^2}+ \|v\W \|_{L^\infty}\|\W' \|_{L^\infty}  \|\W'' \|_{L^2} \Big]\\
    \lesssim&M^3 t^{-\frac{5}{4}-\delta}\log t,\\
    \|v\pv f_4\|_{L^2_v} \lesssim&t^{-2}\log t\lp  \|v\W\|_{L^\infty} \|\W'\|_{L^\infty} \|\W'\|_{L^2} + \|v\W\|_{L^\infty} \|v\W'\|_{L^\infty} \|\W''\|_{L^2} + \|v\W'\|_{L^\infty}^2\|\W'\|_{L^2}\rp\\
    +& t^{-2}\|v\W'''\|_{L^2}\\
    \lesssim& Mt^{-\frac{5}{4}-\delta},\\
     v\pv f_5\approx& t^{-2}(v\W |\W'|^2+ v\W^2 \W'')+t^{-2}\log t [ v\W^4 \W'+ v^2\W^3 (\W')^2+ v^2\W^4 \W''],\\
     \|v\pv f_5\|_{L^2_v}\lesssim&Mt^{-\frac{7}{4}}.
\end{align*}
This shows 
\begin{align*}
    \|B_{4a}\|_{L^2_x}\lesssim Mt^{-\frac{5}{4}-\delta}.
\end{align*}
Next, for $B_{4b},$
\begin{align*}
    \left\|B_{4b}\right\|_{L^2_x}\leq  \left\|\frac{i}{2t}\log t (|\W|^2+v\pv|\W|^2)\right\|_{L^\infty} \|B_3\|_{L^2_x}\lesssim M^2t^{-1+\delta} \cd Mt^{-\frac{5}{4}}.
\end{align*}
Combining these two cases results in
\begin{align*}
    \|B_4\|_{L^2_x}\lesssim Mt^{-\frac{5}{4}-\delta}.
\end{align*}
For $B_5,$ we can see that taking another derivative of $f_1,...,f_5$ adds at most $t^{\frac{1}{2}}$ growth.  Since we have an extra factor of $t^{-1}$ compared to $B_3,$ we end up with
\begin{align*}
      \|B_5\|_{L^2_x}\lesssim Mt^{-\frac{5}{4}-\frac{1}{2}-\delta}.
\end{align*}
Overall, this yields
\begin{align*}
      \|\px Lf\|_{L^2_x}\lesssim Mt^{-\frac{5}{4}}\log t,
\end{align*}
which proves the lemma.

\end{proof}

\bibliography{refs}{}

\begin{thebibliography}{10}

\bibitem{albert.ovidiu2023gsqg}
A.~Ai and O-N. Avadanei.
\newblock Low regularity well-posedness for the generalized surface quasi-geostrophic front equation.
\newblock {\em arXiv preprint arXiv:2311.07551}, 2023.

\bibitem{albert.ovidiu2024sqg}
A.~Ai and O-N. Avadanei.
\newblock Well-posedness for the surface quasi-geostrophic front equation.
\newblock {\em Nonlinearity}, 37(5):055022, 2024.

\bibitem{aiifrimtataru2022two}
A.~Ai, M.~Ifrim, and D.~Tataru.
\newblock Two-dimensional gravity waves at low regularity ii: Global solutions.
\newblock {\em Annales de l'Institut Henri Poincar{\'e} C}, 39(4):819--884, 2022.

\bibitem{bahouri2022global}
H.~Bahouri and G.~Perelman.
\newblock {Global well-posedness for the derivative nonlinear Schr\"odinger equation}.
\newblock {\em Inventiones mathematicae}, 229(2):639--688, 2022.

\bibitem{cite:Bialess12}
H.~A. Biagioni and F.~Linares.
\newblock {Ill-posedness for the Derivative Schr\"odinger and generalized Benjamin-Ono equations}.
\newblock {\em Trans. Amer. Math. Soc.}, 353(9):3649–3659, 2001.

\bibitem{byars2024globaldynamicssmalldata}
A.~Byars.
\newblock {Global dynamics of small data solutions to the derivative nonlinear Schrödinger equation}.
\newblock {\em Communications in Partial Differential Equations}, pages 1--23, 2025.

\bibitem{cite:champeaux}
S.~Champeaux, D.~Laveder, T.~Passot, and P.~L. Sulem.
\newblock {Remarks on the parallel propagation of small-amplitude dispersive Alfvénic waves}.
\newblock {\em Nonlin. Processes Geophys.}, 6:169–178, 1999.

\bibitem{cite:selfsim1}
K.~Fujiwara, V.~Georgiev, and T.~Ozawa.
\newblock {Self-similar solutions to the derivative nonlinear Schr\"odinger equation}.
\newblock {\em J. Differential Equations}, 268(12):7940–7961, 2020.

\bibitem{scattering_nls_1979}
J.~Ginibre and G.~Velo.
\newblock On a class of nonlinear schr\"odinger equations ii: Scattering theory, general case.
\newblock {\em J. Functional Analysis}, 32:33--71, 1979.

\bibitem{HayashiNaumkin2013modified}
Zihua Guo, Nakao Hayashi, Yiquan Lin, and Pavel~I Naumkin.
\newblock Modified scattering operator for the derivative nonlinear schrodinger equation.
\newblock {\em SIAM Journal on Mathematical Analysis}, 45(6):3854--3871, 2013.

\bibitem{harrop2016mod_scat}
B.~Harrop-Griffiths.
\newblock Long time behavior of solutions to the mkdv.
\newblock {\em Communications in Partial Differential Equations}, 41(2):282--317, 2016.

\bibitem{cite:1}
B.~Harrop-Griffiths, R.~Killip, M.~Ntekoume, and M.~Vi\c{s}an.
\newblock {Global well-posedness for the derivative nonlinear Schr\"odinger equation in $L^{2}(\mathbb{R})$ }.
\newblock {\em Preprint available at \href{https://arxiv.org/pdf/2204.12548.pdf}{https://arxiv.org/pdf/2204.12548.pdf}}, 2022.

\bibitem{hayashi_naumkin_modscat_nls_1998}
N.~Hayashi and P.~I. Naumkin.
\newblock Asymptotics for large time of solutions to the nonlinear schrödinger and hartree equations.
\newblock {\em American Journal of Mathematics}, 120(2):369--389, 1998.

\bibitem{hayashi1994modified}
N.~Hayashi and T.~Ozawa.
\newblock {Modified wave operators for the derivative nonlinear Schr{\"o}dinger equation}.
\newblock {\em Mathematische Annalen}, 298(1):557--576, 1994.

\bibitem{cite:NLS}
M.~Ifrim and D.~Tataru.
\newblock {Global bounds for the cubic nonlinear Schr{\"o}dinger equation (NLS) in one space dimension}.
\newblock {\em Nonlinearity}, 28(8):2661, 2015.

\bibitem{ifrimtataruholomorphic}
M.~Ifrim and D.~Tataru.
\newblock Two dimensional water waves in holomorphic coordinates {II}: {G}lobal solutions.
\newblock {\em Bull. Soc. Math. France}, 144(2):369--394, 2016.

\bibitem{ifrimtatarucapillary}
M.~Ifrim and D.~Tataru.
\newblock The lifespan of small data solutions in two dimensional capillary water waves.
\newblock {\em Arch. Ration. Mech. Anal.}, 225(3):1279--1346, 2017.

\bibitem{cite:WavePax}
M.~Ifrim and D.~Tataru.
\newblock {Testing by wave packets and modified scattering in nonlinear dispersive pde's.}
\newblock {\em Transactions of the American Mathematical Society, Series B}, 11, 2024.

\bibitem{jenkins2018global}
R.~Jenkins, J.~Liu, P.~Perry, and C.~Sulem.
\newblock {Global well-posedness for the derivative non-linear Schr{\"o}dinger equation}.
\newblock {\em Communications in Partial Differential Equations}, 43(8):1151--1195, 2018.

\bibitem{jenkins2019derivative}
R.~Jenkins, J.~Liu, P.~Perry, and C.~Sulem.
\newblock {The Derivative Nonlinear Schr\"odinger Equation: Global Well-Posedness and Soliton Resolution}.
\newblock {\em Quart. Appl. Math.}, 78:33--73, 2020.

\bibitem{cite:KaupNew}
D.~J. Kaup and A.~C. Newell.
\newblock {An exact solution for a derivative nonlinear Schr\"odinger equation}.
\newblock {\em J. Mathematical Phys.}, 19(4):798–801. MR 464963, 1978.

\bibitem{cite:WPvisanH16}
R.~Killip, M.~Ntekoume, and M.~Vi\c{s}an.
\newblock {On the well-posedness problem for the derivative nonlinear Schr{\"o}dinger equation}.
\newblock {\em Analysis \& PDE}, 16(5):1245--1270, 2023.

\bibitem{cite:selfsim2}
A.V. Kitaev.
\newblock {Self-similar solutions of the modified nonlinear Schro\"odinger equation}.
\newblock {\em Theor. Math. Phys.;(United States)}, 64(3), 1986.

\bibitem{mjoelhus1974application}
E.~Mjoelhus.
\newblock {Application of the reductive perturbation method to long hydromagnetic waves parallel to the magnetic field in a cold plasma}.
\newblock Technical report, Bergen Univ.(Norway). Dept. of Applied Mathematics, 1974.

\bibitem{cite:physical2}
E.~Mjølhus.
\newblock On the modulational instability of hydromagnetic waves parallel to the magnetic field.
\newblock {\em J. Plasma Phys.}, 16(3):321–334, 1976.

\bibitem{cite:physical3}
A.~Rogister.
\newblock Parallel propagation of nonlinear low-frequency waves in high-$\beta$ plasma.
\newblock {\em Physics of Fluids}, 14(12):2733–2739, 1971.

\bibitem{strauss1974scattering}
W.A. Strauss.
\newblock Nonlinear scattering theory.
\newblock In {\em Scattering Theory in Mathematical Physics: Proceedings of the NATO Advanced Study Institute held at Denver, Colo., USA, June 11--29, 1973}, pages 53--78. Springer, 1974.

\bibitem{cite:Taks12}
H.~Takaoka.
\newblock {Well-posedness for the one-dimensional nonlinear Schr\"odinger equation with the derivative nonlinearity}.
\newblock {\em Adv. Differential Equations}, 4(4):561--580, 1999.

\bibitem{cite:Tao}
T.~Tao.
\newblock {\em Nonlinear dispersive equations: local and global analysis}.
\newblock CBMS regional series in mathematics, 2006.

\bibitem{cite:WP>3/2}
M.~Tsutsumi and I.~Fukuda.
\newblock {On solutions of the derivative nonlinear Schr\"odinger equation. Existence and uniqueness theorem. }.
\newblock {\em Funkcial. Ekvac.}, 23(3):259–277, 1980.

\bibitem{cite:WP>3/2.2}
M.~Tsutsumi and I.~Fukuda.
\newblock {On solutions of the derivative nonlinear Schr\"odinger equation. II}.
\newblock {\em Funkcial. Ekvac.}, 24(1):85–94, 1981.

\bibitem{cite:physical4}
M.~Wadati, H.~Sanuki, K.~Konno, and Y.-H. Ichikawa.
\newblock {Circular polarized nonlinear Alfv\'en waves—a new type of nonlinear evolution equation in plasma physics}.
\newblock {\em Rocky Mountain J. Math.}, 8(1-2):323–331, 1978.

\bibitem{wyller1989conserved}
John Wyller.
\newblock Nonlocal conservation laws of the dnls-equation.
\newblock {\em Physica Scripta}, 40(6):717, 1989.

\end{thebibliography}
\bibliographystyle{plain}

\end{document}